\documentclass[reqno,10pt]{amsart}

\usepackage{amsmath,amssymb,bbm}
\usepackage{enumerate}

\newtheorem{theorem}{Theorem}[section]
\newtheorem{lemma}[theorem]{Lemma}

 \theoremstyle{remark}
\newtheorem{remark}[theorem]{Remark}

\newtheorem{example}[theorem]{Example}

\usepackage{graphicx}

\numberwithin{equation}{section}

 \newcommand{\kX}{\mathcal X}
\newcommand{\kE}{\mathcal E} \newcommand{\kY}{\mathcal Y}
 
\newcommand{\kB}{\mathcal B}

\newcommand{\kC}{\mathcal C} 
\newcommand{\kD}{\mathcal D} \newcommand{\kL}{\mathcal L}
\newcommand{\kO}{\mathcal O} \newcommand{\kH}{\mathcal H}
\newcommand{\kU}{\mathcal U} 
\newcommand{\kI}{ I} 
\newcommand{\kZ}{\mathcal Z}

\newcommand{\interior}{\operatorname{int}}

\newcommand{\kCb}{\kC_{\operatorname{b}}}

\newcommand {\D}{\mathrm D} \renewcommand{\d}{\mathrm d}
\renewcommand{\i}{\mathrm i} \newcommand {\e}{\mathrm e}

\renewcommand{\a}{\mathsf a} \renewcommand{\b}{\mathsf b}
\newcommand{\sS}{\mathsf S}

\DeclareMathAlphabet{\mathpzc}{OT1}{pzc}{m}{it}

\newcommand {\1}{\mathbbm{1}} 
\newcommand {\R}{\mathbb R} 
\newcommand{\N}{\mathbb N} 
\newcommand {\C}{\mathbb C} 
 
\newcommand {\Z}{\mathbb Z} 
\renewcommand{\P}{\mathbb P}
\newcommand {\Q}{\mathbb Q}

\newcommand {\id}{\operatorname{id}} 
\newcommand{\spec}{\operatorname{spec}}

\renewcommand {\Re}{\mathfrak{Re}} 
\renewcommand {\Im}{\mathfrak{Im}}
\newcommand{\nb}{{\operatorname{nb}}}

\newcommand{\norm}[3][{\vphantom 1}]{\lVert #2 \rVert_{#3}^{#1}}
\newcommand{\ip}[3][{\vphantom 1}]{\langle #2 \rangle_{#3}^{#1}}

\begin{document}

\title[Time discretizations of PDEs with non-smooth data] {Runge-Kutta
  time semidiscretizations of semilinear PDEs with non-smooth data}

\author[C. Wulff]{Claudia Wulff} \address[C. Wulff]%
{Department of Mathematics\\
  University of Surrey \\
  Guildford GU2 7XH \\
  UK\footnote{corresponding author, email: c.wulff@surrey.ac.uk}
  }
\author[C. Evans]{Chris Evans} \address[C. Evans]%
 {Department of Mathematics\\
  University of Surrey \\
  Guildford GU2 7XH \\
  UK}

\date{\today}

\begin{abstract}
  We study semilinear evolution equations $ \frac {\d U}{\d t}=AU+B(U)$ posed
  on a Hilbert space $\kY$, where $A$ is normal and generates a
  strongly continuous semigroup, $B$ is a 
smooth nonlinearity
  from $\kY_\ell = D(A^\ell)$ to itself, and  
$\ell \in I \subseteq [0,L]$, $L \geq 
  0$, $0,L \in I$.  In particular the one-dimensional semilinear wave equation and nonlinear
  Schr\"odinger equation with periodic, Neumann and Dirichlet boundary
  conditions fit into this framework. We discretize the evolution equation
with  an
  A-stable Runge-Kutta method in time, retaining continuous space,
  and prove convergence  of order $O(h^{p\ell/(p+1)})$
 for non-smooth initial data $U^0\in \kY_\ell$,  where $\ell\leq p+1$, for a method of classical order $p$, extending
 a result by Brenner and Thom\'ee for linear systems.
 Our  approach is to project the semiflow and numerical method to spectral
  Galerkin approximations, and to balance the projection error with the  error of the time discretization of the 
   projected system. Numerical experiments suggest that our
  estimates are sharp.\\[1ex]

\end{abstract}

\maketitle

\tableofcontents


\section{Introduction}
\label{sec:intro}
 
We study the convergence of a class of A-stable Runge-Kutta time
semidiscretizations of the semilinear evolution equation
\begin{equation}
  \frac{\d U}{\d t} = AU + B(U)
  \label{eqn:see}
\end{equation}
for non-smooth initial data $U(0)=U^0$.  In the examples we have
in mind \eqref{eqn:see} is a partial differential equation (PDE). We assume that (\ref{eqn:see}) is
posed on a Hilbert space $\kY$, $A$ is a normal linear operator  
that generates a strongly continuous semigroup, and that $B$ is smooth on a scale
 of Hilbert spaces $\{\kY_\ell\}_{\ell \in I}$, $I \subseteq [0,L]$, $0,L \in I$,
 as detailed in  condition (B)
below.  Here $\kY_\ell=D(A^\ell)\subseteq \kY$,
$\ell \geq 0$.
Note that condition (B) depends on both, the smoothness properties of the nonlinearity
$B(U)$ and  the  boundary conditions.
Under these assumptions the class of equations we consider
includes the  semilinear wave equation and the nonlinear Schr\"odinger
equation in one spatial dimension with periodic, Neumann and Dirichlet boundary conditions (see Examples
\ref{ex:SWE} - \ref{ex:NSE} below). For an example in three space dimensions see Example \ref{ex:cubicNSE}. We
discretize \eqref{eqn:see} in time by an $A$-stable Runge Kutta
method; the condition of $A$-stability ensures that the numerical
method is well-defined on $\kY$, and is satisfied by a large class of methods
including the Gauss-Legendre collocation methods.

Discretizing in time while retaining a continuous spatial parameter
means that we consider the numerical method as a nonlinear operator on
the infinite dimensional space $\kY$. This leads to several
technicalities, in particular  existence results for the
numerical method $\Psi^h$ as well as the semiflow $\Phi^t$ and
regularity of solutions in both cases  are required to ensure   convergence
results analogous to the finite dimensional case. In \cite{OW10A},
existence and regularity of the semiflow of (\ref{eqn:see}) on a scale
of Hilbert spaces, corresponding results for the numerical method, and
full order convergence of the time semidiscretization for sufficiently smooth
data are studied in detail. 
  We review the relevant results in
Sections \ref{sec:hpde} and \ref{sec:rk_schemes}.

In this paper we consider the effect of non-smooth data on the order
of convergence of the time semidiscretization in this setting. We
 consider an $A$-stable Runge-Kutta method of classical order
$p$ applied to the problem \eqref{eqn:see} with initial data
$U^0\in\kY_\ell$, $\ell \in I$.
The main result we give here, Theorem
\ref{thm:ns_convergence}, shows
that we can expect  order of convergence $\kO(h^q)$ where
$q(\ell) = p\ell/(p+1)$ for $0\leq \ell < p+1$. This corresponds
closely with numerical observation, cf. Figure \ref{fig:impr_swe_2}.
 Given a time $T>0$  we prove the
 above order of convergence for the  time-semidiscretization up to time $T$  for any 
 solution  $U(t)$ of \eqref{eqn:see} with a given $\kY_\ell$ bound. Here $\ell>0$ is such
that  $\ell-k \in I$ for $k=1,\ldots , \lfloor \ell \rfloor$ (the greatest integer $\leq \ell$).  
 It is shown in \cite{OW10A} that for $\ell\geq p+1$ we have full order of
convergence $\kO(h^p)$. 
 
The reduction in order of the method from $p$ to $q$ for $\ell <p+1$ is caused by the
occurrence of unbounded operators in the Taylor expansion of the
one-step error coefficient. Our approach is to apply a spectral
Galerkin approximation to the semiflow of the evolution equation
\eqref{eqn:see}, and to discretize the projected evolution equation in time. This
allows us to bound the size of the local error coefficients in terms
of the accuracy of the projection. By balancing the projection error
with the growth of the local error coefficients we obtain the
estimates of our main result, Theorem \ref{thm:ns_convergence}.

Related results include those of Brenner and Thom\'ee \cite{BT79}, who
consider linear evolution equations $\dot U=AU$ in a more general setting, namely posed on a Banach
space $\kX$, where $A$ generates a strongly continuous semigroup $e^{tA}$ on
$\kX$. They show $O(h^q)$ convergence of A-acceptable rational
approximations of the semigroup for non-smooth initial data $U^0\in
D(A^\ell)$, $\ell = 0,\ldots,p+1$, with $q=q(\ell)= p \ell/(p+1)$ as above, if $\ell> (p+1)/2$
(when $\ell\leq (p+1)/2$ they prove convergence with order $q(\ell)< p \ell/(p+1)$).
Kov\'acs \cite{K07} generalizes  this result to certain intermediate spaces
 with arbitrary $\ell \in [0,p+1]$ and also provides sufficient conditions for when
 $q=q(\ell)= p \ell/(p+1)$ for all $\ell \in [0,p+1]$  (which are satisfied in our setting).

For splitting methods, where the linear part of the evolution equation is evaluated
exactly, a higher order of convergence has been obtained for specific choices of $\ell$
and  specific evolution equations in \cite{LubichNSE} and \cite{Ostermann2014}, see also Example \ref{ex:cubicNSE} below. While splitting methods are very
effective for simulating evolution equations for which the linear
evolution $\e^{tA}$ can easily be computed explicitly, Runge--Kutta
methods are still a good choice when an eigen-decomposition of $A$ is
not available,  for example for the semilinear wave equation in an
inhomogeneous medium, see Example ~\ref{ex:swe-nonhom}.
Moreover, the
simplest example of a Gauss--Legendre Runge-Kutta method, the implicit
mid point rule, appears to have some advantage over split step
time-semidiscretizations for the computation of wave trains for
nonlinear Schr\"odinger equations  because the
latter introduce an artificial instability \cite{Herbst}.

For Runge-Kutta time semidiscretizations of dissipative evolution equations, where $A$ is sectorial,  a  better 
 order of convergence  can be obtained, see  \cite{LeRoux} for the linear case and  \cite{LubichNonSmooth,LO93} 
and references therein for the semilinear case.
 Note that our approach is different from the approach of \cite{LubichNonSmooth,LO93}. 
  In \cite{LubichNonSmooth,LO93} some smoothness of the continuous solution is assumed and from that a (fractional) order of convergence
 is obtained, using the variation of constants formula.  
 The order of convergence obtained in \cite{LubichNonSmooth,LO93} is in general lower than
 in the linear case (where full order of convergence is obtained in the  parabolic case  \cite{LeRoux}), but no extra assumptions on the nonlinearity  $B(U)$ of the PDE  are made.
 In particular in \cite[Theorems 4.1 and 4.2]{LO93} the existence of $(p_s+2)$ time derivatives of the continuous solution $U(t)$ of a semilinear parabolic PDE  \eqref{eqn:see} is assumed, where $p_s$ is the stage order of the method.  This assumption is then used to estimate the  error of the numerical approximation of the inhomogenous part of the variation of constants formula. Here the stage order $p_s$ comes into play. 
  Note that   if the nonlinearity $B(U)$ of the evolution equation \eqref{eqn:see} only  satisfies the standard assumption rather than our assumption (B), i.e., is   smooth on $\kY$ only (so that the
  Hilbert space scale  is trivial with $L=0$) then  the existence of $U'(t)$ can be guaranteed
for $U^0 \in \kY_1$ by semigroup theory \cite{P83}, but
it is not clear whether higher order time derivatives of the solution $U(t)$ of \eqref{eqn:see} exist as assumed in \cite{LO93}  - therefore  in \cite{LO93} also  time-dependent perturbations of \eqref{eqn:see} are considered.
In this paper we instead take the approach of making 
assumptions (namely condition  (B)  on the nonlinearity $B(U)$ of the evolution equation  and
the condition that $U^0 \in \kY_\ell$)  which are straightforward to check 
and guarantee the existence of the time derivatives of the continuous solution $U(t)$ up to order $k\leq \ell$.
We then obtain an order of convergence $O(h^{p\ell/(p+1)})$ of the Runge-Kutta discretization which is identical to the order of convergence in the
 linear case \cite{BT79,K07}. 
 In  \cite[Theorem 2.1]{LubichNonSmooth} some smoothness of the inhomogeneity of the PDE is obtained from the smoothing properties of parabolic PDEs, and this is used
 to prove an order of convergence $h\log h$, without the  assumption of the existence of higher time derivatives of the continuous solution $U(t)$. Here we do not consider parabolic PDEs, so that we cannot use this strategy.

Alonso-Mallo and Palencia \cite{Palencia03} study  Runge-Kutta time discretizations of 
 inhomogeneous linear  evolution equations
 where the linear part creates a strongly continuous semigroup. Similarly as in \cite{LO93}
 they obtain an order of convergence depending on the stage order $p_s$ of the Runge-Kutta method.
 They assume the continuous solution $U(t)$ to be $(p+1)$-times differentiable in $t$, but in their context the condition $U(t) \in D(A^{p-p_s})$, where $p$ is the order of the numerical method,
  is in general not satisfied due to the inhomogeneous terms
 in the evolution equation, and this leads to a loss  in the order of convergence compared to our
 results. Note that in our setting, due to our condition (B) on the nonlinearity, provided $U(0) \in \kY_{p+1}$
  we have $U(t) \in D(A^{p+1}) = \kY_{p+1}$  and  $U(t)$ is $p+1$ times differentiable
   in $t$ (in the $\kY$ norm)  and so we get full order of convergence in this case (see \cite{OW10A}).
  Calvo et al \cite{CalvoEtAl07}
 study Runge-Kutta   quadrature methods for  linear evolution equations $\dot U(t) = A(t)U(t)$ which are well-posed
 and 
prove full order convergence  if the continuous solution $U(t)$ has $p+1$ time derivatives; they also  
 obtain fractional orders of convergence as in \cite{BT79}
for solutions $U(t) \in \kY_\ell$ with $\ell<p+1$.

We proceed as follows: in Section \ref{sec:hpde} we introduce the
class of semilinear evolution equations that we consider in this paper,  give some examples, 
review existence and regularity results
of \cite{P83,OW10A} for the semiflow, and adapt them to the case of
non-integer $\ell$. In
Section \ref{sec:rk_schemes} we introduce a class of $A$-stable Runge-Kutta
methods. We review existence and regularity of these methods   when
applied to the semilinear evolution equation \eqref{eqn:see} and  a
convergence result for sufficiently smooth initial data from
\cite{OW10A}. In Section \ref{sec:spec} we study the stability of the
semiflow and numerical method under spectral Galerkin truncation, and
establish estimates for the projection error.  Lemma~\ref{l.see_proj_err} and
\ref{l.reg-proj-num-method} are established in \cite{OW10B} for integer
values of $\ell$; for
completeness we review the proofs, which also work for non-integer
$\ell$. In Section \ref{sec:traj} we prove    our main result on
convergence   of $A$-stable Runge-Kutta discretizations of semilinear evolution equations
for 
non-smooth initial data. In Section \ref{s.appendix} we generalize our result to 
nonlinearities $B(U)$ which are defined on domains other than balls.
 

\section{Semilinear PDEs on a scale of Hilbert spaces}
\label{sec:hpde}
 
In this section we introduce a suitable functional setting for the
class of equations we subsequently study. 
 We review results from \cite{P83,OW10A} on the
local well-posedness and regularity of solutions of  \eqref{eqn:see} 
and give examples.

For a  Hilbert space $\kX$ we let
\begin{equation*}
  \kB^R_{\kX}(U^0)=\{U\in\kX:\norm{U-U^0}{\kX}\leq R\}
\end{equation*}
be the closed ball of radius $R$ around $U^0$ in $\kX$. We make the following assumptions
on the semilinear evolution equation \eqref{eqn:see}:

\begin{enumerate}[(A)]
\item $A$ is a normal linear operator on $\kY$ that generates a
  strongly continuous semigroup of linear operators $e^{tA}$ on $\kY$
  in the sense of \cite{P83}.
  \label{enum:A}
\end{enumerate}
It follows from assumption (A) that there exists $\omega\in \R$
with  
\begin{equation}
\Re(\spec(A))\leq \omega, \quad  \norm{e^{tA}}{\kY\to\kY}\leq   e^{\omega t},
  \label{eqn:semigroup_bound}
\end{equation}
see \cite{P83}. 
In light of (\ref{enum:A}) we define the
continuous scale of Hilbert spaces $\kY_\ell = D(A^\ell )$, $\ell \geq
0$, $\kY_0=\kY$. Thus the parameter $\ell$ is our measure of smoothness of the
data.  For $m>0$ 
we define $\P_m$ to be the spectral projection of $A$ to $\spec(A) \cap \kB^m_\C(0)$,
 let
$\Q_m=\id-\P_m$
and set $\P=\P_1$, $\Q = \id - \P$. We endow $\kY_\ell$ with the inner
product
\begin{equation}
  \ip{U_1,U_2 }{\kY_\ell}=\ip{\P U_1,\P U_2}{\kY} + \ip{|A|^\ell\Q U_1,|A|^\ell\Q U_2}{\kY},
  \label{eqn:inner_prod}
\end{equation}
which implies
\begin{equation}
 \label{e.A_ell}
\|A^\ell\|_{\kY_\ell \to \kY} \leq 1.
\end{equation}
We  deduce from assumption (A)
that for $u\in\kY$, $\lim_{m\to\infty}\P_mu=u$, 
and from \eqref{eqn:inner_prod} the
estimates 
\begin{equation}
  \norm{A^\ell\P_m U}{\kY}\leq m^\ell\norm{\P_m U}{\kY},\quad
 \|\P_m \|_{ \kY_\ell\to \kY_{\ell+k}} \leq m^k,
\quad \norm{\Q_m U}{\kY}\leq m^{-\ell}\norm{U}{\kY_\ell}
  \label{eqn:proj_est}
\end{equation}
for $\ell \geq 0$, $k\geq 0$, $m\geq 1$.

\begin{remark}
When $\ell$ lies in a discrete set such as  $ \N_0$,  for $\ell>0$  often the inner product
\begin{equation}\label{e.graphIP}
\ip{U_1,U_2}{\ell } =\ip{U_1,U_2}{\kY}+\ip{A^\ell U_1, A^\ell U_2}{\kY}
\end{equation}
is used on $\kY_\ell$ instead of  \eqref{eqn:inner_prod}. For $\ell=0$, for consistency, one defines
$\ip{U_1,U_2}{0} =\ip{U_1,U_2}{\kY}$. 
The reason why we do not use this
inner product here is that \eqref{eqn:inner_prod} is
 continuous in $\ell$ as $\ell \to 0$, but the graph inner product  \eqref{e.graphIP} is not:
 we have $\lim_{\ell \to 0} \ip{U_1,U_2}{\ell } = 2\ip{U_1,U_2}{\kY}= 2\ip{U_1,U_2}{0}$. 
\end{remark}

 To formulate our second assumption, on the nonlinearity $B$,
 we introduce the following notation: for Banach spaces $\kX$, $\kZ$, we denote by $\kE^i(\kX,\kZ)$
the space of $i$-multilinear bounded
mappings from $\kX$ to $\kZ$.  For $\kU \subseteq \kX$ we
  write $\kCb^k(\kU,\kZ)$ to denote the set of $k$ times
continuously differentiable functions $F \colon \interior\kU \to \kZ$ 
 such that $F$ and its derivatives $\D^i F$ are bounded as 
maps from  the interior $\interior\kU$ of $\kU$ to
$\kE^i(\kX,\kZ)$ and extend continuously to the boundary of $\interior\kU$ for $i\leq k$.
We set $\kCb(\kU,\kZ)= \kCb^0(\kU,\kZ)$.
Note that if $\dim \kX=\infty$, there are examples of continuous
 functions $F:\kU\to \kZ$ where $\kU$ is closed
and bounded, which do not lie in $\kCb(\kU,\kZ)$, see e.g. \cite[Remark 2.3]{OW10A}.
In the following for $\ell \in \R$ let $\lfloor\ell\rfloor$ be the
largest integer less than or equal to $\ell$ and $\lceil \ell \rceil$
be the smallest integer greater or equal to $\ell$.
Moreover for $R>0$ and $\ell\geq 0$ we abbreviate
\begin{equation}
 \kB_\ell^{R} = \kB_{\kY_\ell}^R(0).
\end{equation}
We are now ready to formulate our condition on the nonlinearity $B(U)$ of
\eqref{eqn:see}.

\begin{enumerate}[(B)]
\item There exists  $L\geq 0$, $I \subseteq [0,L]$,
$0,L \in I$,
$N\in \N$, $N > \lceil L\rceil $,  such that $B\in
  \kCb^{N- \lceil \ell \rceil }(\kB^R_\ell;\kY_\ell)$ for
all  $\ell\in I$ and $R>0$.
  \label{enum:B}
\end{enumerate}
We denote the supremum  of
$B:\kB_\ell^R\to\kY_\ell$ as $M_\ell[R]$ and the supremum of its   derivative as
 $M'_\ell[R]$,  and set $M[R] =M_0[R] $ and $M'[R] =M'_0[R]$.  
Moreover we define
\begin{equation}\label{e.I}
I^-:= \{ \ell \in I, \ell -k\in I, k =1,\ldots, \lfloor\ell \rfloor\}.
\end{equation}
We seek a solution $U(\cdot)\in
\kC([0,T];\kY_\ell)$ of (\ref{eqn:see}) for some $T>0$, $\ell \in I$, with initial data $U(0)=U^0\in\kY_\ell$, and write
$\Phi^t(U^0)\equiv\Phi(U^0,t)\equiv U(t)$. 
The following result is an extension  of Theorem 2.4 
  of
\cite{OW10A}, see also \cite{P83},
to non-integer $\ell$ and 
provides well-posedness and regularity of the semiflow $\Phi^t$ under
suitable assumptions.

\begin{theorem}[Regularity of the semiflow]\label{t.semiflow} 
  Assume that the semilinear evolution equation \eqref{eqn:see}
  satisfies (A) and (B). Let $R>0$. Then
  there is $T_*>0$ such that there
  exists a semiflow $\Phi$ 
which satisfies
\begin{subequations}
  \begin{equation}
    \Phi^t\in \kCb^{N}(\kB_0^{R/2};\kB^R_0) 
   \label{e.PhiUDeriv}
  \end{equation}
with uniform bounds in  $t\in [0,T_*]$. Moreover if $\ell \in I^-$ and 
 $k\in \N_0$ satisfies
$k \leq  \ell$,  then
\begin{equation}
 \Phi(U) \in \kCb^k([0,T_*];\kB^R_0) 
 \label{eqn:semiflow_regt}
\end{equation}
with uniform bounds in $U \in \kB_\ell^{R/2}$.
\label{e.semiflow_reg}
\end{subequations}
  The bounds on $T_*$ and   $\Phi$    depend only on $R$,  
  $\omega$ from \eqref{eqn:semigroup_bound}, and the bounds  afforded by
  assumption (B) on balls of radius $R$.
\end{theorem}

\begin{proof}
The proof of \eqref{e.semiflow_reg} is an application of a contraction mapping theorem with
parameters to  the map
\begin{equation}
  \Pi(W,U,T) =e^{tTA}U  + \int_0^t e^{T(t-\tau)A}B(W(\tau)) \d \tau,
  \label{eqn:see_mild}
\end{equation}
on the scale of Banach spaces
$\kZ_{\ell} = \kCb([0,1];\kY_{\ell})$,   $\ell \in I$, where we define $\kZ:= \kZ_0$.  
The solution $W(U,T)(t) =\Phi^{tT}(U)$  of \eqref{eqn:see} is obtained as a fixed point  of \eqref{eqn:see_mild} for   $U \in \kB^{R/2}_\kY(0)$
as in \cite{OW10A}. Here $\Pi:\kB^{R}_{\kZ}(0) \times \kB^{R/2}_\kY(0)\times [0,T_*]\to
   \kZ$.
  In order to apply the contraction mapping theorem we first check that $\Pi(W,\cdot,\cdot)$ maps $\kB^{R}_{\kZ}( 0)$ to itself: 
For   $U\in \kB^{R/2}_\kY(0)$    we have
  \begin{align}
    \|\Pi(W,U,T) \|_{\kZ} &\leq
\max_{\tau \in [0,1]} \|\e^{\tau TA} U\|_{\kY} 
 + T \e^{\omega T} M_0[R]  \label{e.ProblemTerm}\\
    &\leq  \e^{\omega T}R/2
    + T \e^{\omega T} M_0[R]
    \leq R\notag
  \end{align}
  for $T\in [0,T_*]$ and $T_*$ small enough.
So $\Pi$ maps $\kB^{R}_{\kZ}( 0)$ to itself. Moreover for sufficiently small $T_*$ there is $c\in [0,1)$ such
that   $\|\D\Pi(W,U,T)\|_{\kZ\to \kZ}\leq c$ for all $W \in \kB^R_{\kZ}(0)$, $U \in \kB^{R/2}_\kY(0)$ and 
$T \in [0,T_*]$ so that
$\Pi$ is a
  contraction. Hence, $W\in
  \kCb(\kB^{R/2}_{\kY}( 0) \times [0,T_*];\kB^{R}_{\kZ}( 0))$
with $N$ derivatives in the first component. 
This proves statements 
\eqref{e.PhiUDeriv} and also $\Phi(U) \in \kCb^k([0,T_*];\kB_0^R)$  in the case $k=0$. 

For  $k \in \N$, $k\leq \ell$ it follows from the fact $\ell \in I^-$  that the
  above argument  applies with $\kY$ replaced by $\kY_{\ell-j}$,  $j=0,\ldots,
  k$. Hence there is some $T_*>0$ such that  $\Phi \in
  \kCb(\kB^{R/2}_{\ell-j}\times[0,T_*];\kB^{R}_{\ell-j})$ for $j=0,\ldots,
  k$. As detailed in \cite{OW10A} for $U \in \kB^{R/2}_{\ell}$ the 
$t$ derivatives up to order $k$ can then be obtained by
  implicit differentiation of $\Pi(W(U,T),U,T)=W(U,T)$ with $\Pi$ defined
  above which implies that $\Phi(U) \in \kCb^k([0,T_*];\kB^R_0)$  for $k\leq \ell$ with uniform bounds in
$U \in \kB^{R/2}_{\ell}$.
\end{proof}

 Note that this theorem extends to mixed $(U,t)$ derivatives
which are, however, in  general only strongly   continuous in $t$, see \cite{OW10A}
for details. For our purposes in this paper the above theorem is sufficient.

\begin{example} [Semilinear wave equation, 
 periodic boundary conditions] \label{ex:SWE} Consider the
  semilinear wave equation
  \begin{equation}
    \label{e.swe}
    \partial_{tt}u = \partial_{xx} u - V'(u)
  \end{equation} 
  on $[0,2\pi]$ with periodic boundary conditions. Writing
  $v=\partial_tu$ and $U=(u,v)^T$ Equation \eqref{e.swe} takes the
  form \eqref{eqn:see} where
  \begin{equation}
    \label{e.AB}
    A =\Q_0 \tilde A, \quad \tilde A =
    \begin{pmatrix}
      0 & \id \\ \partial^2_x & 0
    \end{pmatrix} \,, \qquad B(U) = {0 \choose -V'(u)} + \P_0 \tilde A
    U.
  \end{equation}
Here $\P_0$ is the spectral projector of $\tilde A$ to the eigenvalue $0$.
  Since the Laplacian is diagonal in the Fourier representation with
  eigenvalues $-k^2$ for $k\in\Z$, the eigenvalue problem for $A$
  separates into $2\times 2$ eigenvalue problems on each Fourier mode,
  and it is easy to see that the spectrum of $A$ is given by
  \[
  \spec A = \{ \i k \colon k \in \Z\} \setminus \{0\}.
  \]
  Note that $\P_0 \tilde A$ has a Jordan block and is hence included
  with the nonlinearity $B$.  We denote the Fourier coefficients of a
  function $u \in \kL^2([0,2\pi];\R^d)$ by $\hat{u}_k$, so that
  \begin{equation}
    u(x) = \frac{1}{\sqrt{2\pi}} 
    \sum_{k\in\Z} \hat{u}_k \, {\e}^{{\i}kx} \,.
    \label{e.fourier}
  \end{equation}
  Then the Sobolev space $ \kH_{\ell}([0,2\pi];\R^d)$ is the Hilbert space of all $u
  \in \kL^2([0,2\pi];\R^d)$ for which
  \[
  \norm[2]{u}{\kH_{\ell}} = \langle u , u \rangle_{\kH_{\ell}} <
  \infty \,,
  \]
  where the inner product is given by
  \begin{equation}
    \label{e.inner-product_HEll}
    \langle u, v \rangle_{\kH_{\ell}} = \langle\hat{u}_0 , \hat{v}_0 \rangle_{\R^d}+ \sum_{k\in \Z} |k|^{2\ell} \, \, \langle \hat{u}_k , \hat{v}_k \rangle_{\R^d}.
  \end{equation}
   In the setting of the
  semilinear wave equation, we have
  \begin{equation}
   \label{eq:Yrl}
  \kY_{\ell} = \kH_{\ell+1}([0,2\pi];\R) \times \kH_{\ell}([0,2\pi];\R) \,,
  \end{equation}
 and   the group $e^{tA}$ is unitary on any $\kY_{\ell}$. So (A) is
  satisfied.  Moreover in this example, the inner product \eqref{eqn:inner_prod} on $\kY_\ell$ corresponds to   the inner product
   defined via \eqref{e.inner-product_HEll}.   If the potential $V:\R\to \R$ is analytic, then, by
  Lemma~\ref{l.nonlinearityH^l} a) below, the nonlinearity $B(U)$ is
  analytic as map of $\kY_{\ell}$ to itself for any $\ell\geq 0$ and $B$ and its derivatives are  bounded on balls around $0$.
  Hence  assumption (B) holds for any $L\geq 0$ 
and $N> \lceil L \rceil$ with $I=[0,L]$.  
\end{example}

\begin{example}[Semilinear wave equation, 
 non-analytic nonlinearity] \label{ex:SWEnonAnalytic}
If $V\in \kC^{N+2}(\R)$   then (B)  holds with  $I = [0,L]$ and $\lceil L\rceil < N$.  To see this note that     Lemma \ref{l.nonlinearityH^l} c)
applied to $f = V' \in \kC^{N+1}(\R)$    ensures that $f \in \kCb^{N-\lfloor \ell \rfloor}(\kB_{\kH_{\ell+1}}^R;
\kH_{\ell})$  for all $R>0$ and therefore that (B) holds, noting that $\kY_\ell$ is as in \eqref{eq:Yrl}. Here we abbreviated 
$\kH_{\ell}:=\kH_{\ell}([0,2\pi];\R)$.
\end{example}

\begin{example} [Semilinear wave equation, Dirichlet boundary conditions]  
\label{ex:SWE-Dirichlet} 
When endowed with homogeneous Dirichlet boundary 
conditions $u(t,0) = u(t,\pi)=0$  the linear part  $A $  of the semilinear wave 
equation \eqref{e.swe} still generates a unitary group. In this case we have
$\P_0=0$, $A=\tilde A$, and
\[
\kY_\ell =D(A^\ell) =  \kH^0_{\ell+1} ([0,\pi];\R) \times \kH_\ell^0 ([0,\pi];\R).
\]
Here  $\kH_\ell^0 ([0,\pi];\R) = D((-\Delta)^{\ell/2})$, where
$\Delta$ denotes the Laplacian with Dirichlet boundary conditions.
 By  \cite{Fujiware} for $\ell \notin 2  \N_0+\frac12 $
\[
 \kH_\ell^0 ([0,\pi];\R) = \{ u \in \kH_\ell ([0,\pi];\R): u^{(2j)}(0)=u^{(2j)}(\pi)=0
~~\mbox{for}~~0\leq 2j< \ell-\frac12 \}.
\]
If $V:\R\to \R$ is analytic and even so that $f=-V'$ satisfies the required boundary conditions, the conclusions of
Lemma \ref{l.nonlinearityH^l} a) apply to $f=-V'$ on  the spaces $\kH_{\ell+1}^0 ([0,\pi];\R)$  and 
$\kH_{\ell}^0 ([0,\pi];\R)$,  provided  that $\ell +1\notin   \frac{1}2 + 2\N_0$ or 
$\ell\notin   \frac{1}2 + 2\N_0$, respectively. Since we need $-V'(u)$ to map from an open
set of $\kH_{\ell+1}^0 ([0,\pi];\R)$ into $\kH_{\ell}^0 ([0,\pi];\R)$ it is sufficient to satisfy
either of those two  constraints on $\ell$,  at least one of which is always true.
So in this example 
condition (B) is   satisfied with $I = [0,L]$ for any $L\geq 0$. Moreover the condition that 
$V$ is even may be relaxed to the requirement that 
 $V^{(2j+1)}(0)=0$ for $0\leq 2j\leq  L+\frac12$. 
\end{example}

\begin{example} [Semilinear wave equation, Neumann boundary conditions]  
\label{ex:SWE-Neumann} In the case of Neumann boundary conditions on $  [0,\pi]$,    the operator $A=\tilde A$
from \eqref{e.AB} is again skew-symmetric and 
  has the same spectrum as in Example \ref{ex:SWE}.  In
this case, $\kY_\ell = \kH_{\ell+1}^\nb( [0,\pi];\R ) \times \kH_{\ell}^\nb( [0,\pi];\R )$. Here 
$\kH_\ell^\nb ([0,\pi];\R) = D((-\Delta)^{\ell/2})$, where
$\Delta$ now denotes the Laplacian with Neumann boundary conditions.
Due to \cite{Fujiware}
\[
  \kH_{\ell}^\nb([0,\pi];\R ) 
  = \{u \in \kH_{\ell}([0,\pi];\R ) \colon u^{(2j+1)}(0) = u^{(2j+1)}(\pi)=0
    \text{ for } 0\leq 2 j <\ell   -\frac{ 3}2  \} \,,  
\]
for $\ell \notin 3/2+2\N_0$.
If $V:\R\to\R$ is analytic, then the conclusions of Lemma \ref{l.nonlinearityH^l} a) apply to $f=-V'$ on  the spaces $\kH_{\ell+1}^\nb ([0,\pi];\R)$ ($\kH_{\ell}^\nb ([0,\pi];\R)$) whenever
 $\ell +1\notin   \frac{3}2 + 2\N_0$ ($\ell\notin   \frac{3}2 + 2\N_0$).  This follows from the fact that all terms in the sum
obtained from computing $\partial_x^{2j+1}f(u)$  contain at least one odd derivative of $u$ of order
at most $2j + 1$, so that the required boundary conditions for $f$ are satisfied.
Hence
Condition (B) is   satisfied for any $L\geq 0$ with $I =  [0,L] $.
\end{example}

\begin{example}[A semilinear wave equation in an inhomogeneous
material]
\label{ex:swe-nonhom}
Instead of \eqref{e.swe}, let us consider the non-constant coefficient
semilinear wave equation
\[ 
  \partial_{tt} u = \partial_x (a \, \partial_x u) + b \, u -V'(u)
\]
 with   periodic boundary conditions
where $V\in \kC^{N+2}(\R)$, $a,b \in \kCb^N([0,2\pi];\R)$ are $2\pi$-periodic 
with $a(x) >0$ and $b(x) \leq 0$ for $x
\in [0,2\pi]$. Then the  conclusions of Example \ref{ex:SWEnonAnalytic} apply.
\end{example}

\begin{example}[Nonlinear Schr\"odinger equation] \label{ex:NSE}
  Consider the nonlinear Schr\"odinger equation
  \begin{equation}
    \label{e.nse}
    \i \, \partial_t u 
    =  \partial_{xx} u + \partial_{\bar{u}} V(u,\bar{u})
  \end{equation}
  on $[0,2\pi]$ with periodic boundary conditions, where $V(u,\bar{u})$
  is assumed to be analytic as a function in $u_1=\Re\,(u)$ and $u_2=\Im\,(u)$.  Setting $U
  =(u_1, u_2)$, we can write \eqref{e.nse} in the form \eqref{eqn:see} with
  \begin{equation}
    \label{e.nlsdefs}
    A = \left( \begin{array}{cc} 0 &  \partial^2_x\\
    -\partial_x^2 &  0 
    \end{array}
    \right) \,,  \quad 
    B(U) = \frac12 { \partial_{u_2} V \choose - \partial_{u_1} V} .
  \end{equation}
  The Laplacian is diagonal in the Fourier representation
  \eqref{e.fourier} with eigenvalues $-k^2$ and
  $\kL^2([0,2\pi];\C)$-orthonormal basis of eigenvectors $\e^{\pm \i k
    x}/\sqrt{2\pi}$ where $k \in \Z$. Hence, the spectrum of $A  $ is given by
  \[
  \spec A = \{ -\i k^2 \colon k \in \Z\}
  \]
  and $A$ is normal and generates a unitary group on $\kL^2([0,2\pi];\C)$
  and, more generally, on every $\kH_{\ell}([0,2\pi];\C)$ with $\ell \geq
  0$.

  By Lemma~\ref{l.nonlinearityH^l} a)  below the nonlinearity $B(U)$
  defined in \eqref{e.nlsdefs} is analytic as map from $\kH_{\ell}([0,2\pi];\R^2)$ to
  itself for every $\ell>1/2$.  Hence, assumption (B) holds for the
  nonlinear Schr\"odinger equation \eqref{e.nse} for any $I=[0,L]$, $L\geq 0$ if we
  set $\kY_\ell = \kH_{2\ell+\alpha}([0,2\pi];\R^2)$ for $\alpha>1/2$. 

When we equip the nonlinear Schr\"odinger equation \eqref{e.nse}
with Dirichlet (Neumann)  boundary conditions
we need to require that $\ell +\frac\alpha 2 \notin \N_0 + \frac{1}4$ ($\ell +\frac\alpha 2 \notin \N_0 + \frac{3}4$)
and, for Dirichlet boundary conditions,  we need the potential $V$
 to be even  or satisfy 
$V^{(2j+1)}(0)=0$ for $0\leq j< L+\alpha-\frac{1}4$.
Here $I = [0,L] \setminus (\N_0 + \frac{1}4-\frac\alpha 2)$ for Dirichlet boundary conditions
and $I = [0,L] \setminus (\N_0 + \frac{3}4-\frac\alpha 2)$ for Neumann boundary conditions.
\end{example}

The nonlinearities of the PDEs in the above examples are  superposition operators $f \colon \kH_{\ell}([0,2\pi];\R^d) \to
\kH_{\ell}([0,2\pi];\R^d)$ of smooth functions $f:D \subseteq \R^d \to \R^d$ or restrictions of
such operators to  spaces encorporating boundary conditions.  To prove that these superposition operators
satisfy assumption (B) we have employed the
following lemma. Part a) of this lemma has already been stated in slightly different
form in \cite{Ferrari98,Matthies01}, and parts b) and c) follow from \cite{OW10A}.

\begin{lemma}[Superposition operators]\label{l.nonlinearityH^l}
Let $\Omega \subseteq \R^n$ be an open set satisfying the cone property. 
\begin{itemize}
\item[a)]  
 Let  $\rho>0$ and let $f \colon \kB_{\C^d}^\rho \to \C^d$ be analytic. If $\Omega$ is unbounded assume
   $f(0)=0$. Then $f$ is also analytic as a
  function from $\kB_{\kH_\ell}^R$ to  $\kH_{\ell}:= \kH_{\ell}(\Omega;\C^d)$  for every $\ell >
  n/2$ and $R\leq \rho/c$ with $c$ from  \eqref{eq:uvHm} below. Moreover $f:\kB_{\kH_\ell}^R \to \kH_\ell$ and its derivatives up to order $N$ are bounded with $N$-dependent bounds  for arbitrary $N\in \N$.
  \item[b)]
  Let  $f\in \kCb^{N}(D,\R^d)$  for some open  set $D \subset \R^d$ and $N \in \N$. If $\Omega$ is unbounded assume
   $f(0)=0$.  Let  $j \in \N$ be such that $j> n/2$. Let $\kD$ be an   $\kH_j$ bounded subset of 
\[
\{ u \in \kH_j(\Omega;\R),~ u(\Omega) \subset D \}
\]
and  for  $R>0$, $k\in \N$ with $k\geq j$ let
\begin{equation}
\label{e.Dk-superpos}
\kD_k = \kD \cap
 \kB^R_{\kH_k}(0).
\end{equation}
  Here  $\kH_k =\kH_{k}(\Omega;\R^d) $. Then,
 \[
 f \in \kCb^{N-k}(\kD_k; \kH_k), \quad\mbox{for}\quad k \in  \{j,\ldots, N\}
 \]
    with $R$-dependent bounds.
  \item[c)]
  Let  $D$, $f$ and $j$ be as in b)  and   let $L> n/2$ be such that $\lfloor L\rfloor \leq  N$. Then 
\[
f  \in \kCb^{N-\lfloor \ell \rfloor}(\kD_{\ell}; \kH_{\ell-1}) \quad\mbox{for all}\quad \ell \in [j,L], 
\]
with $\kD_\ell$ defined as in \eqref{e.Dk-superpos}.
  \end{itemize}

\end{lemma}
\begin{proof}
We restrict to the case $d=1$. A generalization to $d>1$ is straightforward.

To prove a)
 let $\ell>n/2$. Then there exists a constant $c=c(\ell)$ such that
  for every $u, v \in \kH_{\ell}(\Omega;\C)$ we have $uv \in
  \kH_{\ell}(\Omega;\C)$ with
  \begin{equation}\label{eq:uvHm}
    \norm{uv}{\kH_{\ell}}
    \leq c \, \norm{u}{\kH_{\ell}(\Omega;\C)} \, \norm{v}{\kH_{\ell}(\Omega;\C)} \,,
  \end{equation}
  see, e.g., \cite{Adams}. Let $f$ be analytic on $\kB_{\C}^\rho$ and let
  \begin{equation}
    \label{e.fTaylor}
    f(z) = \sum_{n=0}^\infty \, a_n \, z^n \,
  \end{equation}
  be the Taylor series of $f$ around $0$ for $|z|\leq \rho$. Let $g \colon
  \R\to \R$ be its majorization
  \[
  g(s) = \sum_{n=0}^\infty \, \lvert a_n \rvert \, s^{n} \,.
  \]
  By applying the algebra inequality \eqref{eq:uvHm} to each term of
  the power series expansion \eqref{e.fTaylor} of $f(u)$, we see that
  the series converges for every $u \in \kH_{\ell}$ provided $\ell >
  n/2$, and that
  \begin{equation}
    \norm{f(u)}{\kH_{\ell}} \leq c^{-1} \, g \bigl( c \, \norm{u}{\kH_{\ell}} \bigr)  + |a_0|(\sqrt{|\Omega|} - c^{-1})\,,
    \label{e.majorant}
  \end{equation}
  where $c$ is as in \eqref{eq:uvHm}, $R\leq \rho/c$  and $a_0=0$ if $\Omega$ is unbounded. In
  other words, $f$ is analytic and bounded as function from a ball of radius $R$ around $0$ in  $\kH_{\ell}= \kH_{\ell}(\Omega;\C)$
  to $\kH_{\ell}$. Similarly we see that the
  same holds for the derivatives of $f$.
  
    To prove b)  note that $\kD$ is well-defined because by the Sobolev embedding theorem
  $\kH_j(\Omega;\R) \subseteq \kCb(\Omega;\R)$. In
 \cite[Theorem 2.12]{OW10A}, the statement was proved in the case $n=1$. The extension to the case
 $n>1$ is straightforward. Here let us just illustrate the idea of the proof for the  example $n=1$, $N=1$ and  $j=k=1$. Then 
$f \in \kCb^1(\kD_1;\kL_2)$
 by the Sobolev embedding theorem, but also  $f \in \kCb(\kD_1;\kH_1)$ since for this we only need that
$\partial_x f(u) = f'(u) \partial_x u \in \kL_2$ with uniform bound in $u \in\kD_1$ 
 which is again true by the Sobolev embedding theorem. 
 
 To prove c) note that
   for $\ell \in [j,L]$  we know from b) that
$f  \in \kC^{N-\lfloor \ell \rfloor}(\kD_{\lfloor \ell \rfloor}; \kH_{\lfloor\ell \rfloor})$. Since 
$\kD_{\ell} \subseteq \kD_{\lfloor \ell \rfloor}$  
and $\kH_{\lfloor\ell \rfloor} \subseteq \kH_{ \ell -1}$  this implies
$f  \in \kC^{N-\lfloor \ell \rfloor}(\kD_{\ell}; \kH_{\ell -1})$.
\end{proof}


\section{Runge-Kutta time semidiscretizations}
\label{sec:rk_schemes}

In this section we apply an A-stable Runge-Kutta method in time to the evolution
equation \eqref{eqn:see}, and establish well-posedness and regularity
 of the numerical method on the infinite dimensional space
$\kY$.

Given an $(s,s)$ matrix $\a$, and
a vector $\b\in\R^s$, we define the corresponding Runge-Kutta method
by
\begin{subequations}
  \begin{align}
    &W = U^0\1 + h\a (AW+B(W)),\label{eqn:rk_stage}\\
    &\Psi^h(U^0) = U^0+h\b^T(AW+B(W))\label{eqn:rk_step},
  \end{align}
  \label{eqn:rk_both}
\end{subequations}
where
\begin{equation*}
  U \1 = 
  \begin{pmatrix}
    U \\
    \vdots \\
    U
  \end{pmatrix}\in\kY^s ~~\mbox{ for }~~U \in \kY,\quad
  W =
  \begin{pmatrix}
    W^1 \\
    \vdots \\
    W^s
  \end{pmatrix},\quad B(W) =
  \begin{pmatrix}
    B(W^1) \\
    \vdots \\
    B(W^s)
  \end{pmatrix}.
\end{equation*}
Here, $W^1,\ldots,W^s$ are the stages of the method, we understand $A$
to act diagonally on the vector $W$, i.e., $(AW)^i=AW^i$, and
\begin{equation*}
  (\a W)^i=\sum_{j=1}^s\a_{ij}W^j,\quad \b^TW = \sum_{i=1}^s \b_iW^i.
\end{equation*}
We define
\[
\|W\|_{\kY_\ell^s} := \max_{j=1,\ldots,s} \|W^i\|_{\kY_\ell}
\]
and re-write (\ref{eqn:rk_stage}) as
\begin{equation}\label{e.Pi}
  W  = (\id- h\a A)^{-1}(\1 U^0 + h\a B(W)),
\end{equation}
and (\ref{eqn:rk_step}) as
\begin{equation}\label{e.psi}
  \Psi(U,h)=\Psi^h(U) = \sS(hA)U+h\b^T(\id-h\a A)^{-1}B(W(U,h)),
\end{equation}
where $\sS$ is the stability function, given by
\begin{equation}
  \sS(z) = 1 + z\b^T(\id-z\a)^{-1}\1.
  \label{eqn:stability_function}
\end{equation}
In the following 
$\C^-_0=\{z\in\C:\Re(z)\leq 0\}$.  We assume
$A$-stability of the numerical method as follows (cf.~\cite{LO93}):
\begin{enumerate}[(RK1)]
\item $\sS(z)$ from (\ref{eqn:stability_function}) is   bounded with
  $|\sS(z)|\leq 1$ for all $z\in\C^-_0$.
  \label{enum:RK1}
\item $\a$ is invertible and the matrices $\id-z\a$ are invertible for all $z\in\C^-_0$.
  \label{enum:RK2}
\end{enumerate}

\begin{example}
   Gauss-Legendre collocation methods such the implicit midpoint rule satisfy (RK1) and (RK2)
  \cite[Lemma 3.6]{OW10A}.
  \label{ex:GL_methods}
\end{example}

The following result is needed later on, see also
\cite[Lemmas 3.10, 3.11, 3.13]{OW10A}:

\begin{lemma} 
\label{l.ShA}
  Under assumptions (A), (RK1) and (RK2) there are $h_*>0$,
  $\Lambda>0$  and  $\sigma>0$  such that for $h\in
  [0,h_*]$
  \begin{subequations}
    \begin{gather}
      \norm{\sS(hA)}{\kY\to\kY}\leq 1+\sigma h \label{e.ShA} \\
      \norm{(\id-h\a A)^{-1}}{\kY^s\to\kY^s}\leq \Lambda
\label{e.Lambda}.
    \end{gather}
  Moreover, for any $k\in\N_0$,  $U \in \kY_k$, $W \in \kY_k^s$,  
  \[
  h\mapsto \sS(hA)U  \in \kCb^k([0,h_*]; \kY) ,
\]
and
\[ h \mapsto (\id-h\a A)^{-1}W \in \kCb^k([0,h_*]; \kY^s), \quad h\mapsto h(\id-h\a A)^{-1}W\in \kCb^{k+1}([0,h_*]; \kY^s).
  \]
  Finally
 there are $c_{\sS,k} >0$ with 
\begin{equation}
\label{e.csk}
\sup_{h \in [0,h_*]} \|\partial_h^k \sS(hA) \|_{\kY_k\to\kY} \leq c_{\sS,k},
\end{equation}
and, with $\Lambda_k:=k! \|\a\|^{k} \Lambda^{k+1}$, we have for $k\in \N_0$,
\begin{equation}
\label{e.Lambdak}
  \|\partial^k_h( (\id-h\a A)^{-1} ) \|_{\kY^s_{k} \to \kY^s} \leq
  \Lambda_k, \quad
 \|\partial^k_h( h(\id-h\a A)^{-1} ) \|_{\kY^s_{k-1} \to \kY^s} \leq
  \Lambda_k/\|\a\|.
  \end{equation}
  \label{e.lemShA}
  \end{subequations}
\end{lemma}
\begin{proof}
Most of the statements follow directly from \cite[Lemmas 3.10, 3.11, 3.13]{OW10A}.
\eqref{e.Lambdak} follows from 
\[
  \partial_h^k(\id - h \a A)^{-1} 
  = k! \, (\a A)^k\, (\id - h \a A)^{-k-1} \,.
\]
and   
\[
  \partial_h^k [ h (\id - h \a A)^{-1}] = 
  \partial_h^{k-1} (\id - h \a A)^{-2} =  
  k! \, (\a A)^{k-1} \, (\id - h \a A)^{-k-1} \,,
\]
see \cite[Lemma 3.10]{OW10A}.
 \end{proof}

Analogously to Theorem \ref{t.semiflow}, we require a
well-posedness and regularity result for the stage vectors $W^i$, $i=1,\ldots, s$, and
the numerical method $\Psi^h$. The following result is an extension of  \cite[Theorem
3.14]{OW10A}  to non-integer values of $\ell$.

\begin{theorem}[Regularity of numerical method]
  Assume that the semilinear evolution equation (\ref{eqn:see})
  satisfies (\ref{enum:A}) and (B), and apply a Runge-Kutta method
  subject to conditions (RK1) and (RK2).  Let $R>0$.
 Then there is $h_*>0$ such that  there
  exist a stage vector $W$ and numerical method $\Psi$ 
which satisfy
  \begin{subequations}
\begin{equation}\label{eqn:glob-U-num_method_regularity}
    W^i(\cdot,h), \Psi(\cdot,h) \in \kCb^{N} (\kB_0^r;\kB^R_0)
  \end{equation}
 for $i=1,\ldots, s$, where 
 \begin{equation}
    \label{e.R*}
  r= r(R) = \frac{R}{2 \Lambda}.
  \end{equation}
with uniform bounds in  $h \in [0,h_*]$.
 Furthermore, for
  $\ell \in I^-$, $k\in \N_0 $, $k \leq  \ell$,
we have for $i=1,\ldots, s$,
\begin{equation}\label{eqn:glob-num_method_regularity}
    W^i(U,\cdot), \Psi(U,\cdot) \in \kCb^k( [0,h_*];\kB_0^R)
  \end{equation}
\end{subequations}
  with uniform bounds in $U \in\kB^r_\ell$. 
  The bounds on   $h_*$,  $\Psi$ and $W$   depend only on  $R$,
   \eqref{e.lemShA}, those afforded by assumption (B) on balls of radius $R$ and
 on $\a$, $\b$  as specified by the
  numerical method.
  \label{thm:num_method_regularity}
\end{theorem}

\begin{proof} 
As in \cite{OW10A} we   compute $W$ as
  fixed point of the map $\Pi:\kB^R_{\kY^s}(0) \times\kB^r_\kY(0) \times [0,h_*]\to
   \kY^s$, given by
  \begin{equation}
  \label{e.Pi-W}
  \Pi(W,U,h) = (\id-h \a A)^{-1} \1 U + h \a (1-h \a A)^{-1} B(W),
  \end{equation}
  using \eqref{e.Pi}.
  To be able to apply the contraction mapping theorem 
  we need to check that $\Pi(W,U,h) \in  \kB^{R}_{\kY^s}(0)$ for $U \in \kB^{r}_\kY(0)$. 
For  such $U$  we have
  \begin{align}
    \|\Pi(W,U,h) \|_{\kY^s} &\leq
\|(\id-h \a A)^{-1}\1  U\|_{\kY^s} + h  \norm{\a}{}\Lambda M \notag\\
    &\leq  \Lambda r
    + h \norm{\a}{}\Lambda M \leq R/2 + h \Lambda \norm{\a}{}M
    \leq R \label{e.Pi-Fix}
  \end{align}
  for $h\in [0,h_*]$ and $h_*$ small enough, with $M = M_0[R]$.
So $\Pi$ maps $\kB^{R}_{\kY^s}(0)$ to itself. Furthermore there is some $c \in [0,1)$ such that
$\|\D\Pi(W,U,h)\|_{\kY^s \to \kY^s} \leq c$  for $W\in \kB^R_{\kY^s}(0)$,  $W\in \kB^r_\kY(0)$, $h\in [0,h_*]$ if  $h_*$ is small enough,
 and so $\Pi$ is a
  contraction. Hence, $W\in
  \kCb(\kB^{r}_{\kY}(0)) \times [0,h_*];\kB^{R}_{\kY^s}(0))$
with $N$ derivatives in $U$. 

  This proves statements
\eqref{eqn:glob-U-num_method_regularity} and also
\eqref{eqn:glob-num_method_regularity}  in the case $k=0$  for $W$. 
Due to \eqref{e.psi}, these statements
also hold true for $\Psi$. 
In the case $k\neq 0$ it follows from the
  that $\ell \in I^-$ that the above argument also
  holds on $\kY_{\ell-j}$, $j=0,\ldots,
  k$. Hence  there is some $h_*>0$  such that  $W^i, \Psi \in
  \kCb(\kB^r_{\ell-j}\times[0,h_*];\kB^R_{\ell-j})$, $j=0,\ldots,
  k$, $i=1,\ldots, s$. As shown in \cite{OW10A} for $U \in\kB^r_{\ell}$ the 
$h$ derivatives up to order $k$ can then be obtained by
  implicit differentiation of $\Pi(W,U,h)=W(U,h)$ with $\Pi$ defined
  above and by differentiating \eqref{e.psi},  cf.~the proof of
  Theorem \ref{t.semiflow}. This then implies
  \eqref{eqn:glob-num_method_regularity}.
\end{proof}

A discretization $y^{n+1}=\psi^h(y^n)$ of an ordinary differential equation (ODE) $\frac{\d y}{\d t}= f(y)$ is
said to be of classical order $p$ if the local error, i.e., the one-step error, of the numerical 
method  is given by the Taylor remainder of order $p+1$,
\begin{equation}\label{e.le}
  y(h)-\psi^h(y^0) = \int_0^h\frac{(h-\tau)^p}{p!}\partial_\tau^{p+1}(y(\tau)-\psi^\tau(y^0))\d\tau.
\end{equation}
When considering the local error of a semidiscretization of a  
PDE on a Hilbert space $\kY$, the derivatives of the semiflow and
numerical method in time and step size respectively are not
necessarily defined on the whole space $\kY$. To obtain global error
estimates for semidiscretizations of PDE problems analogous to the
familiar results for ODEs, we must consider the local error as a map
$\kZ\to\kY$, where $\kZ$ is a space of higher regularity.  Using the
regularity results for the semiflow and its discretization in time,
Theorems \ref{t.semiflow} and \ref{thm:num_method_regularity}, the
following can be shown (see \cite[Theorem 3.20]{OW10A}): if
(A), (B), (RK1) and (RK2) hold, and 
  (in our notation) $\ell \in I^-$, $\ell\geq p+1 $
then for fixed $T>0$, $R>0$ there exist constants
$c_1,c_2,h_*>0$ such that for every solution
$\Phi^t(U^0)$, $t\in[0,T]$ with
$\| \Phi^t(U^0)\|_{\kY_{p+1}}
 \leq R$ and every
$h\in[0,h_*]$, we have
\begin{equation}
  \label{eq:smooth_convergence}
  \norm{\Phi^{nh}(U^0)-(\Psi^h)^n(U^0)}{\kY}\leq c_1e^{c_2nh}h^p,
\end{equation}
provided that $nh\leq T$.   In
this paper we study the case where the solution $U(t)$
satisfies $U(t) \in \kY_\ell$ with $\ell<p+1$, by means of Galerkin truncation.


\section{Spectral Galerkin truncations}
\label{sec:spec}

In this section we consider the stability of the semiflow $\Phi^t$ of
\eqref{eqn:see}, and the numerical method $\Psi^h$ defined by
\eqref{eqn:rk_both} under truncation to a Galerkin subspace of
$\kY$. As before for  $m>0$ we denote by $\P_m$  the spectral projection
operator of $A$ on to the set $\spec(A)\cap \kB^{m}_{\C}(0)$, and set
$\Q_m=\id-\P_m$. 
In this setting we define $B_m(u_m)=\P_mB(u_m)$, and consider the
projected semilinear evolution equation
\begin{align}
  \frac{ \d u_m}{\d t}&= Au_m + B_m(u_m) \label{eqn:see_proj}
\end{align}
with flow map $\phi_m^t(u^0_m)=u_m(t)$ for $u_m(0)=u^0_m \in
\P_m \kY$.  Moreover we define $\Phi^t_m :=\phi^t_m\circ \P_m$.
The  Galerkin truncated semiflow has the same regularity properties as the full
semiflow (see Theorem \ref{t.semiflow})  uniformly in $m$.

\begin{lemma}[Regularity of projected semiflow]\label{l.proj-semiflow} 
  Assume   (A) and (B) and let $R>0$. Then
  there is $T_*>0$ such that for $m\geq 0$ there
  exists a projected semiflow $\Phi_m$ 
which satisfies
\begin{subequations}
  \begin{equation}
    \Phi^t_m\in \kCb^{N}(\kB_0^{R/2};\kB_0^R) 
   \label{e.PhimUDeriv}
  \end{equation}
with uniform bounds in  $t\in [0,T_*]$ and $m\geq 0$. Moreover if $\ell \in I^-$  and 
 $k\in \N_0$ satisfies
$k \leq \ell$,  then
\begin{equation}
 \Phi_m(U) \in \kCb^k([0,T_*];\kB_0^R) 
 \label{eqn:semiflowm_regt}
\end{equation}
with uniform bounds in $U \in \kB^{R/2}_\ell$ and $m\geq 0$.
\label{e.proj-semiflow_reg}
\end{subequations}
  The bounds on $T_*$ and $\Phi_m$,   depend only on $R$, 
  $\omega$ from \eqref{eqn:semigroup_bound}, and those afforded by
  assumption (B) on balls of radius $R$.
\end{lemma}

  In the case $B\equiv 0$ it is clear that for $U^0 \in \kY_\ell$ 
we have the estimate
$\norm{\Phi^t(U^0)-\Phi_m^t(U^0)}{\kY} = \kO(m^{-\ell})$ on any
finite interval of existence $[0,T]$. With the presence of a nonlinear perturbation
$B\neq 0$ a similar result can be obtained by a Gronwall type argument
as shown in the lemma below, which  gives an appropriate bound for the error of the semiflow incurred in Galerkin
truncation. Note that similar  results for mixed higher order derivatives in
time and initial value are obtained, for integer $\ell$
in \cite[Theorems 2.6 and 2.8]{OW10B}.

\begin{lemma}[Projection error for the semiflow]\label{l.see_proj_err}
  Assume that the semilinear evolution equation \eqref{eqn:see}
  satisfies (A) and (B), let  
 $\ell>0$, $T>0$ and $\delta>0$. Then for all $U^0$ with
 \begin{subequations}
 \begin{equation}\label{e.PhiGalErrorCond}
    \|\Phi^t(U^0)\|_{\kY_\ell} \leq R, ~~t \in [0,T]
  \end{equation}
  there is $m_*\geq 0$ such that for $m\geq m_*$ we have
  $\Phi_m^t(U^0) \in \kB^{R+\delta}_0$ for $t\in [0,T]$, and
  \begin{equation}\label{e.globPhiGalError}
    \| \Phi^t(U^0)-\Phi^t_m(U^0) \|_\kY = m^{-\ell}R  \e^{(\omega + M')t}= \kO(m^{-\ell})
  \end{equation}
  for $m\geq m_*$ and
$t\in [0,T]$, where $M'=M'_0[R+\delta]$.
\end{subequations}
  Here $m_* $ and the order constant  depend only on $\delta$, $R$, $T$,
  \eqref{eqn:semigroup_bound} and the bounds afforded by
 (B) on balls of radius $R+\delta$.
\end{lemma}
\begin{proof}
  The statement is
  shown for integer $\ell$ in \cite{OW10B}.  We review the argument, which  also works for
  arbitrary $\ell \in I$. To prove
  \eqref{e.globPhiGalError} we use the mild formulation
  \eqref{eqn:see_mild} for $\Phi$ and $\Phi_m$.  We find
  \begin{align*}
    \norm{\Phi^t(U^0)-\Phi_m^t(U^0)}{\kY} &\leq \norm{ \Q_m\Phi^t(U^0)}{\kY} \\
& +
    \norm{ \int_0^te^{(t-\tau)A}(\P_m B(\Phi^\tau(U^0))-\P_mB(\Phi^\tau_m(U^0)))\d \tau
    }{\kY}
    \\
    &\leq     m^{-\ell} R  +   \int_0^{t} \e^{\omega(t-\tau)} \|B(\Phi^\tau(U^0))-B(\Phi^\tau_m(U^0))\|_\kY \d \tau
    \\
    &\leq  m^{-\ell}R+  M'
    \int_0^{t}   \e^{\omega(t-\tau)}\norm{\Phi^\tau(U^0)-\Phi_m^\tau(U^0)}{\kY}\d \tau,
  \end{align*}
  where $M' = M'_0[R+\delta]$ a
  bound of $\D B$ as map from $\kB_0^{R+\delta}$ 
 to  $\kE(\kY)$, see condition (B), and we choose $m_*>0$ big enough such that
 \begin{equation}
 \label{e.phimphiDelta}
 \norm{\Phi^\tau(U^0)-\Phi_m^\tau(U^0)}{\kY} \leq \delta \quad\mbox{for}\quad \tau \in [0,T].
 \end{equation} Thus, applying a Gronwall
  type argument, we obtain \eqref{e.globPhiGalError}.
\end{proof}

We also consider an $s$-stage Runge-Kutta method applied to the
projected semilinear evolution equation \eqref{eqn:see_proj}. We
denote by $w_m=w_m(u^0_m,h)$ the stage vector of this map, and by
$\psi^h_m(u_m^0)$ the one-step numerical method applied to the
projected system \eqref{eqn:see_proj} and define $W_m=w_m\circ\P_m$,
$\Psi^h_m=\psi^h_m\circ\P_m$. Similar to Lemma \ref{l.proj-semiflow} and Lemma 
\ref{l.see_proj_err}, we have the following results regarding the
existence, regularity and error under truncation for the projected
numerical method. 
Note that  similar results have been obtained, for integer $\ell$, and mixed derivatives
in \cite[Theorems  3.2 and 3.6]{OW10B}.

\begin{lemma}[Regularity of projected numerical method
and projection error]
\label{l.reg-proj-num-method}
  Assume that the semilinear evolution equation (\ref{eqn:see})
  satisfies (A) and (B), and apply a Runge-Kutta method
  subject to conditions (RK1) and (RK2).  Let $R>0$.
 Then there is $h_*>0 $ such that for $m\geq 0$  there
  exist a stage vector $W_m$ and numerical method $\Psi_m$ of the projected
system \eqref{eqn:see_proj} 
which satisfy
  \begin{subequations}
\begin{equation}\label{eqn:U-proj-num_method_regularity}
    W^i_m(\cdot,h), \Psi_m(\cdot,h) \in \kCb^{N} (\kB^{r}_0;\kB^R_0)
  \end{equation}
for $i=1,\ldots, s$, where $r$ is as in \eqref{e.R*},
with uniform bounds in  $h \in [0,h_*]$, $m\geq 0$.
 Furthermore, for
  $\ell \in I^-$,  $k\in \N_0 $, $k \leq \ell$,
we have for $i=1,\ldots, s$,
\begin{equation}\label{eqn:h-proj-num_method_regularity}
    W^i_m(U,\cdot), \Psi_m(U,\cdot) \in \kCb^k( [0,h_*];\kB^R_0)
  \end{equation}
with uniform bounds in $U \in\kB^r_\ell$, $m\geq 0$. 
Finally,  if  $\ell \in \kI$, $\ell>0$,  then for
$m\geq 0$ we get 
  \begin{equation}\label{e.wwm}
    \sup_{ \substack{  U\in\kB^r_\ell
        \\ h\in[0,h_*]  }    }
    \norm{W(U,h)-W_m(U,h)}{\kY^s}=\kO(m^{-\ell})
  \end{equation}
and
 \begin{equation}\label{e.psipsim}
    \sup_{ \substack{  U\in \kB^r_\ell
        \\ h\in[0,h_*]  }    }
    \norm{\Psi(U,h)-\Psi_m(U,h)}{\kY}=\kO(m^{-\ell}).
  \end{equation}
\end{subequations}
 The bounds on   $h_*$,  $\Psi_m$ and $W_m$   and the order constants 
depend only on $R$, 
   \eqref{e.lemShA}, those afforded by assumption (B) on balls of radius $R$ and
 on $\a$, $\b$  as specified by the
  numerical method.
\end{lemma}

\begin{proof}
  The statements \eqref{eqn:U-proj-num_method_regularity} 
and \eqref{eqn:h-proj-num_method_regularity} are shown exactly as 
in the proof of Theorem \ref{thm:num_method_regularity} 
and   \eqref{e.wwm}, \eqref{e.psipsim} are
  shown for integer $\ell$ in \cite{OW10B}. The same arguments are valid for arbitrary
  $\ell \in I$ as well, we review the proof for completeness. 
From  the formulation (\ref{e.Pi}) of the stage vectors $W^i,W^i_m$,
  $i=1,\ldots,s$,  
we   find 
  \begin{align}
    \norm{W(U,h)-W_m(U,h)}{\kY^s} &\leq
    \norm{(\id-h\a A)^{-1}}{\kY^s\to\kY^s}\norm{\Q_mU}{\kY} \notag\\
 & \quad + 
\norm{h \a (\id-h\a A)^{-1} \Q_mB(W)}{\kY^s}  \notag\\
    & \quad 
+h\norm{(\id-h\a A)^{-1}}{\kY^s\to\kY^s}\norm{\a}{}\norm{\P_m(B(W)-B(W_m)))}{\kY^s}  \notag\\
    &\leq \Lambda \norm{\Q_m U}{\kY}  + h \|\a\| \Lambda m^{-\ell}M_\ell[R]
     \notag \\
& \quad+ h \Lambda \norm{\a}{} \norm{\P_m B(W(U,h))-\P_m B(W_m(U,h))}{\kY^s}
 \notag   \\
    & \leq \Lambda
    \norm{U}{\kY_\ell}m^{-\ell}+  h \|\a\| \Lambda m^{-\ell}M_\ell[R]
   \notag  \\
    & \qquad+ h \Lambda \norm{\a}{} M' \norm{W(U,h)-W_m(U,h)}{\kY^s} \label{e.WWm-est}
  \end{align}
with an order constant uniform in $U\in \kB^r_{\ell}$. Here $M' = M'_0[R]$ and  
 we used \eqref{e.Lambda}  and  \eqref{eqn:proj_est}.
 Solving for
  $\norm{W(U,h)-W_m(U,h)}{\kY^s}$ and taking the supremum over
  $h\in[0,h_*]$ and $U \in \kB^r_{\ell}$ 
we get \eqref{e.wwm}.

Similarly for the numerical method using \eqref{e.psi},  \eqref{e.lemShA} and \eqref{eqn:proj_est} we estimate 
  \begin{align}
   \norm{\Psi^h(U)-\Psi^h_m(U)}{\kY} & \leq  
 \norm{\sS(hA) }{\kY\to\kY}\norm{\Q_mU}{\kY}  +   \| \Q_m \b h(\id - h \a A)^{-1} B(W)\|_{\kY^s} \notag \\
& \quad  
+ h \|\b\| \Lambda\norm{\P_m(B(W)- B(W_m))}{\kY^s} \notag \\
   & \leq   (1+\sigma h)\norm{U}{\kY_\ell}m^{-\ell}  + 
s\|\b\| h \Lambda m^{-\ell} M_\ell[R] \notag   \\
&\quad + s h\norm{\b}{}\Lambda M'\norm{W(U)-W_m(U)}{\kY^s}
 \notag \\
   & \leq   (1+\sigma h)\norm{U}{\kY_\ell}m^{-\ell}  +s \|\b\| h \Lambda m^{-\ell} M_\ell[R]
\notag   \\
&\quad + s h\norm{\b}{}\Lambda M' \kO(m^{-\ell}).
\label{e.psipsimHelp}
  \end{align}
  Here we used \eqref{e.wwm} in the last line.
\end{proof}


\section{Trajectory error bounds for non-smooth data}
\label{sec:traj}
 
In this section we consider the convergence of the global error
\begin{equation}
  E^n(U,h)=\norm{\Phi^{nh}(U)-(\Psi^h)^n(U)}{\kY}
  \label{eqn:ge}
\end{equation}
as $h\to 0$ for non-smooth initial data.  As
mentioned above, cf. \eqref{eq:smooth_convergence}, \cite[Theorem
3.20]{OW10A} states that we have $E^n(U^0,h)=\kO(h^p)$ in some
interval $[0,T]$, $0\leq nh\leq T$, 
given sufficient regularity of the semiflow and time
semidiscretization to bound the local error given by the Taylor
expansion to order $p+1$ as a map
\begin{equation}
  U\mapsto\norm{\int_0^h\frac{(h-\tau)^p}{p!}\partial_\tau^{p+1}(\Phi^\tau(U^0)-\Psi^\tau(U^0))\ d\tau}{\kY},
  \label{eqn:le}
\end{equation}
see \eqref{e.le}.  As stated by Theorems \ref{t.semiflow} and
\ref{thm:num_method_regularity}, this is the case provided $\ell \in I^-$, 
 $\ell\geq
p+1$.  In this paper we study the order $q=q(\ell)$ of convergence of
the global error for non-smooth initial data $U^0 \in \kY_\ell$, $\ell
\in I^-$, $\ell<p+1$, such that $E^n(U,h)=\kO(h^q)$ and show that we
obtain $q(\ell) = p\ell/(\ell+1)$ as Brenner and Thom\'ee \cite{BT79}
and Kov\'acs \cite{K07} 
did for  linear  strongly continuous semigroups.

The implicit midpoint rule, the simplest Gauss-Legendre method,
satisfies the conditions (RK1) and (RK2), see Example
\ref{ex:GL_methods} with $p=2$.  Figure \ref{fig:impr_swe_2} shows the order of
convergence of the implicit midpoint rule applied to the semilinear
wave equation (\ref{e.swe}) with $V'(u)=u-4u^2$ for   $\ell=j/2$, $j=0,\ldots, 6$, on the integration
interval $t \in [0,0.5]$,
using a fine spatial mesh (we use $N=1000$ grid points on $[0, 2\pi]$). 
As initial values we choose $U^0= (u^0, v^0) \in \kY_\ell$ where
\[
u^0(x) =
 \sum_{k=0}^{N=1} \frac{c_u}{k^{\ell + 1/2+ \epsilon}} (\cos kx
+   \sin kx),
\quad
v^0(x)  =
 \sum_{k=0}^{N=1} \frac{c_v}{k^{\ell + 1/2+ \epsilon}} (\cos kx
+ \sin kx).
\]
Here $c_u$ and $c_v$ are such that $\|U^0\|_{\kY_\ell}= 1$, with $U^0 = (u^0, v^0)$, 
and  $\epsilon = 10^{-8}$. From Theorem \ref{t.semiflow}, with $\kY$ replaced by $\kY_\ell$, we know
that there is some $T_*>0$ such that $\Phi^t(U^0) \in \kB^R_\ell$  for $U^0 \in \kY_\ell$ so that the
assumption \eqref{e.condPhiGlobal} of our convergence  result, Theorem \ref{thm:ns_convergence} below, is satisfied. 
We   integrate the semilinear wave equation with the above initial data for the time steps  $h = 0.1, 0.095, 0.09, 0.085,\ldots, 0.05$,  when $\ell>0$. At $\ell=0$, to reduce computational effort, we only used   the time steps  $h = 0.1,  0.09, \ldots 0.05$.
To estimate the trajectory error,  we compare the numerical solution  to a solution
calculated using a much smaller time step, $\tilde h = 10^{-3}$ for $\ell>0$ and $\tilde h = 10^{-4}$ for $\ell=0$.  
From the assumption 
$E_n(h) = c h^{q}$
we get $\log E_n(h) = \log c + q \log h$. Fitting a line to those data, we take the gradient
of the line as our estimated order of convergence of the trajectory error.
The decay in $q(\ell)$ as $\ell$ decreases
from $3$ is clearly shown. Note that the order of convergence does not decrease to exactly $0$ at $\ell=0$ and is slightly better than predicted by our theory when $\ell=2.5$. This is  because we simulate a space-time discretization rather than a time semidiscretization. Moreover at $\ell=0$, despite the fact that we already use a finer time  step size, the approximation of the exact solution is not  that accurate as the order of convergence for the time-semidiscretization vanishes at $\ell=0$. 

\begin{figure}[h]
  \begin{center}
  \vspace*{-6cm}
  \includegraphics[scale=0.6]{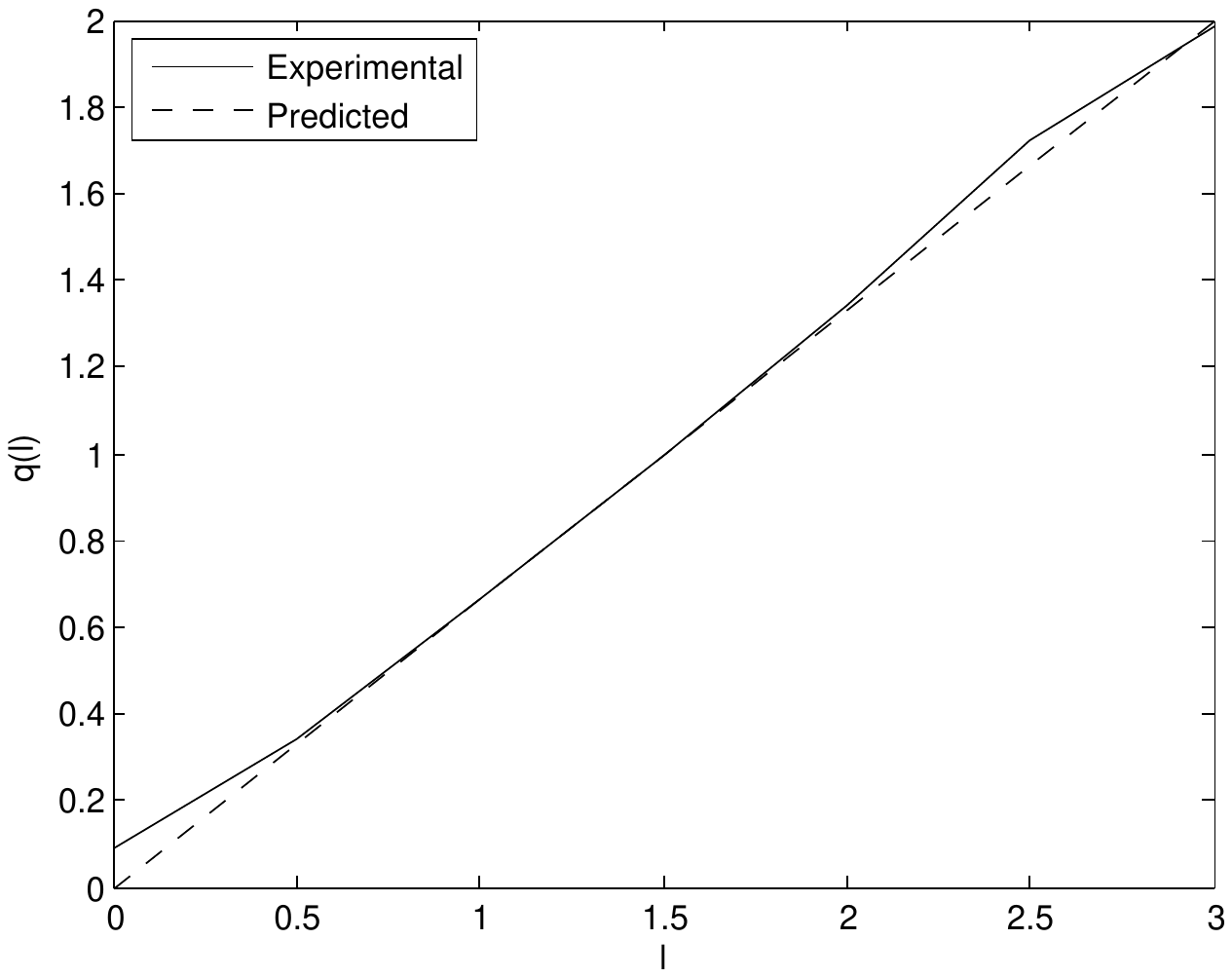}
  \end{center}
    \vspace*{-6cm}
  \caption{Plot of a numerical estimate of $q(\ell)$ against $\ell$
    for the implicit midpoint rule applied to the semilinear wave
    equation, with the prediction of Theorem~\ref{thm:ns_convergence}
    for comparison.}
  \label{fig:impr_swe_2}
\end{figure}

In the rest of this section, equipped with the results of Section \ref{sec:spec} on the
stability of the semiflow and the numerical method under Galerkin,
truncation we  estimate the growth with $m$ of the local
error of a Runge-Kutta method (\ref{eqn:rk_both}), subject to (RK1)
and (RK2), applied to the projected equation (\ref{eqn:see_proj})
subject to (\ref{enum:A}) and (B) for non-smooth initial data. In this
setting, by coupling $m$ and $h$ and balancing the projection error and 
trajectory error of the projected system, we   obtain an estimate
for $q(\ell)$ that describes the convergence of the numerical method for the
semilinear evolution equation \eqref{eqn:see} as observed in Figure
\ref{fig:impr_swe_2}, see Section \ref{ss.err_nonsmooth}.


\subsection{Preliminaries}\label{ss.Prelim}
We start with some preliminary lemmas.

\begin{lemma}[$m$-dependent bounds for derivatives of $\Phi_m$]\label{l.phimDerivs}
Assume that the semilinear evolution equation \eqref{eqn:see}
  satisfies (A) and (B) and 
choose 
$\ell\in I^-$, $T>0$, $m_*\geq 0$ and $R>0$.
 Then  for all $U^0$ with 
  \begin{subequations}
\begin{equation}\label{e.glob-cond-proj_flow_reg}
  \Phi^t_m(U^0) \in \kB_\ell^R\quad\mbox{for}\quad t\in [0,T],~~m\geq m_*,
  \end{equation}
  and for all $k\in\N_0$, $k\leq  \ell$ we have
  \begin{equation}
    \Phi_m(U^0)  \in \kC_b^{k}( [0,T];\kB_0^R) 
    \label{eqn:globproj_flow_reg}
  \end{equation}
  with bounds uniform in $U^0$ and $m\geq m_*$. 
Further,  choose $k\in \N_0$ with
    $\ell\leq k \leq N$. Then for all $U^0$ satisfying
  \eqref{e.glob-cond-proj_flow_reg},
\eqref{eqn:globproj_flow_reg} still holds, but 
  with $m$-dependent bounds which are uniform in $U^0$. Moreover for all such
  $U^0$, $\ell\leq k \leq N$,
  \begin{equation}
    \norm{\partial_t^k\Phi^t_m(U^0)}{\kCb([0,T];\kY )}=
    \kO(m^{k-\ell}), 
\label{eqn:proj_flow_m_dep}
  \end{equation}
\end{subequations}
  with bounds uniform in $U^0$.   The bounds and order
constants only depend on $T$,    $R$, \eqref{eqn:semigroup_bound}
and the bounds from assumption (B).
\end{lemma}

\begin{proof}
 Due to Lemma \ref{l.proj-semiflow} statement \eqref{eqn:globproj_flow_reg} 
is non-trivial only if $\ell \geq 1$. In this case let  $u_m(t) = \Phi_m^t(U^0)$.
From Lemma \ref{l.proj-semiflow} with $\kY$ replaced by $\kY_\ell$, using  \eqref{eqn:globproj_flow_reg} 
we also get $u_m \in \kCb([0,T];{\kB^R_{\ell}})$. From 
\eqref{e.glob-cond-proj_flow_reg} and  \eqref{eqn:see_proj} we conclude
that $ \partial_t u_m \in \kCb([0,T];{\kY_{\ell-1}})$ and thus,  $ u_m \in \kCb([0,T];{\kY_\ell})  \cap \kCb^1([0,T];{\kB_{\ell-1}^R})$
with bounds uniform in  $m\geq m_*$ and $U^0$ satisfying
 \eqref{e.glob-cond-proj_flow_reg}.  That proves   \eqref{eqn:globproj_flow_reg} for $k=1$.
If $\ell \geq 2$ then from \eqref{eqn:see_proj} we  get
$\partial_t u_m \in \kCb^1([0,T];{\kY_{\ell-2}})$ and therefore 
$  u_m \in \kCb^2([0,T];{\kB_{\ell-2}^R})$.
 Inductively this proves that  
\begin{equation}
\Phi_m(U^0)\in \kCb^{k}([0,T]; \kB_{\ell-k}^R)
\label{eqn:proj_flow_k_m_indep}
\end{equation}
for $ k\leq \ell$ 
  with uniform bounds in $m\geq m_*$ and in all $U^0$ satisfying
  \eqref{e.glob-cond-proj_flow_reg}.
This proves \eqref{eqn:globproj_flow_reg} for $k \leq \ell$ with $m$ independent
bounds.   

To prove
  \eqref{eqn:proj_flow_m_dep} we proceed by induction over $k =
  \lceil \ell \rceil,\ldots,N$.
We consider the cases $\ell<1$ and $\ell \geq 1$ separately.
 If $\ell<1$ then from \eqref{eqn:see_proj} we have 
 \[
  \|\partial_t u_m\|_{\kC([0,T];\kY)} \leq \|A_m u_m\|_{\kC([0,T];\kY)}+  M 
 \leq m^{1-\ell} \|u_m\|_{\kC([0,T];\kY_\ell)}+ M =O(m^{1-\ell}) 
 \]
 where $M=M_0[R]$,
with order constant 
independent of $m\geq m_*$ and of  $U^0$ satisfying
 \eqref{e.glob-cond-proj_flow_reg}. 
This then immediately shows 
 \eqref{eqn:proj_flow_m_dep} for $k= \lceil \ell \rceil=1$.
 If $\ell\geq 1$, $\ell \in\Z$ then the start of the induction is $k=\ell$, and the left hand side of \eqref{eqn:proj_flow_m_dep}
 is bounded by \eqref{eqn:globproj_flow_reg}.
 
 If $\ell\geq 1$, $\ell \notin\Z$
 then the start
  of the induction is $k = \lceil \ell \rceil > \ell$. Using
 \eqref{eqn:proj_flow_k_m_indep}   we can bound
  the $\lfloor \ell \rfloor$-th derivative independent of $m$ in the
  $\kY_{\ell-\lfloor \ell \rfloor}$ norm. Using the Fa\`a di Bruno
  formula \cite{FaadiBruno} we find that for any $i \in \N$,  $i< N$,
  \begin{align}
    \partial_t^{i+1}u_m &= \partial_t^i(Au_m+B_m(u_m))\notag \\
    &= A(\partial_t^iu_m)+ \sum_{1\leq \beta \leq i}\frac{i!
      \D^\beta_uB_m(u_m)}{j_1!\cdots j_i!} 
    \prod_{\alpha=1}^i\left(\frac{\partial_t^\alpha
        u_m}{\alpha!}\right)^{j_\alpha},
    \label{e.FaaDiBruno}
  \end{align}
  where  
 $\beta=j_1+\cdots+j_i$ and the    sum is over all $j_\alpha \in N_0$, $\alpha = 1,\ldots, i$, with 
  $j_1+2j_2+\cdots +  ij_i=i$. We consider \eqref{e.FaaDiBruno} with $i$
  replaced by $\lfloor \ell\rfloor$. Then the second term in the last
  line of \eqref{e.FaaDiBruno} is bounded independent of $m\geq m_*$
  due to \eqref{eqn:globproj_flow_reg}. Furthermore, since
  $\partial_t^{\lfloor\ell\rfloor}u_m\in\kY_{\ell-\lfloor\ell\rfloor}$  by \eqref{eqn:proj_flow_k_m_indep}  with   uniform bound in $m\geq m_*$,
  we estimate
  \[
  \|A(\partial_t^{\lfloor\ell\rfloor}u_m)\|_\kY
  =\|A^{1+\lfloor\ell\rfloor-\ell}(A^{\ell-\lfloor\ell\rfloor}
  \partial_t^{\lfloor\ell\rfloor}u_m)\|_\kY=\kO(m^{1+\lfloor\ell\rfloor-\ell}),
  \]
  where we have used    the first
  inequality of (\ref{eqn:proj_est}). So \eqref{eqn:proj_flow_m_dep}
also  holds true for $k =i+1= \lceil\ell\rceil$ when $\ell> 1$, $\ell \notin \Z$. 

Now fix an integer $k$ and assume that
  \eqref{eqn:proj_flow_m_dep} holds for all integers $i$ such
  that $\ell\leq i \leq k$. We now  use
  \eqref{e.FaaDiBruno} with $i=k$ to estimate $\|  \partial_t^{k+1}u_m\|_\kY$. By the first inequality of
  (\ref{eqn:proj_est}) and the induction hypothesis the first term on
  the second line of \eqref{e.FaaDiBruno}  is
  $\kO(m^{k+1-\ell})$. Moreover, by \eqref{eqn:globproj_flow_reg} and the induction hypothesis, the
  $\kY$ norm of the second term is of order $\kO(m^n)$ with
$n=0$  if $j_{\lceil \ell \rceil} + \ldots + j_k=0$  and
$$
n= (\lceil \ell \rceil -\ell)j_{\lceil \ell \rceil}+\cdots+
(k-\ell)j_k\leq k-\ell.
$$
if  $j_{\lceil \ell \rceil} + \ldots + j_k>0$.
Thus we see  that the right hand term  of \eqref{e.FaaDiBruno}, with $i=k$, is $\kO(m^{k+1-\ell})$
as well.
\end{proof}

\begin{lemma}[$m$-dependent bounds for derivatives of $\Psi_m$ and $W_m$]
  Assume that the semilinear evolution equation \eqref{eqn:see}
  satisfies (A) and (B), and apply a Runge-Kutta method
  subject to (RK1) and (RK2). Choose 
 $\ell\in I^-$ and
$k\in \N_0$ with 
  $\ell \leq k\leq N$.
  Let $R>0$ and define $r$ as in \eqref{e.R*}.
Then there is
  $h_*>0$   such that  for $m\geq 0$
and    $i=1,\ldots, s$, 
  \begin{equation}
    W_m^i(U,\cdot),\Psi_m(U,\cdot) \in \kC_b^{k}([0,h_*];\kB_0^R)
    \quad\mbox{
      for}   \quad i=1,\ldots,s
    \label{e.wmpsim}
  \end{equation}
  with $m$-dependent bounds which are uniform in 
$U \in \kB_\ell^r  $.
Moreover 
\begin{equation}\label{e.deriv-wm-psim-bound}
    \sup_{ \substack{  U\in\kB_\ell^r      \\ h\in[0,h_*] }}
\norm{\partial_h^{k}\Psi^h_m(U)}{\kY} = \kO(m^{k-\ell}),
    \quad
    \sup_{ \substack{  U\in \kB_\ell^r 
        \\ h\in[0,h_*]  }    }\norm{\partial_h^{k}W_m(U,h)}{\kY^s} = \kO(m^{k-\ell}).
  \end{equation}
The  order constants in  \eqref{e.deriv-wm-psim-bound} depend only $R$,  \eqref{e.lemShA},
$\a$ and $\b$ from the numerical method 
and the bounds afforded by (B) on balls of radius $R$. 
 \label{lem:psim-m-dep-deriv}
\end{lemma}

\begin{proof}
 By Lemma
  \ref{l.reg-proj-num-method}, with $\kY$ replaced by $\kY_{\ell-j}$,
  \begin{equation}
    \Psi_m, W_m^i \in \kC_b^{j} ([0,h_*]; \kB^R_{\ell-j}),
    \label{eqn:proj_psi_m_indep}
  \end{equation}
for $i=1,\ldots, s$, $j=1\ldots, \lfloor\ell\rfloor$,
  with bounds independent over $m\geq 0$ and 
$U \in \kB^r_\ell$.  From \eqref{e.psi} we
formally  obtain 
  \begin{equation}\label{e.dkhPsi_m}
  \partial_h^k\Psi_m^h(U,h) = \partial^k_h\sS(hA)\P_m U+ \sum_{j=0}^k
{k \choose j}  \b^T \partial^{k-j}_h ( h(\id-h\a A)^{-1}) \partial_h^j \P_m
  B(W_m(U,h)).
  \end{equation}
  By \eqref{e.csk}  and    \eqref{eqn:proj_est} there are $h_*>0$, $c_{\sS,k}$ such that
for all  $h\in
  [0,h_*]$ and $k\geq \ell$
  \begin{equation}\label{e.Sm-est}
    \|\partial^k_h\sS(hA)\P_m \|_{\kY_\ell \to \kY} \leq
    \|\partial^k_h\sS(hA)  \|_{\kY_k \to \kY} \|\P_m  \|_{\kY_\ell \to \kY_k}
    \leq  c_{\sS,k} m^{k-\ell}.
  \end{equation}
 In addition \eqref{e.Lambdak} shows that for $n\in \N$ with $n-1 \geq \ell$
  \begin{align}
    \|\partial^n_h( h(\id-h\a A)^{-1} \P_m)\|_{\kY^s_\ell \to \kY^s} &
    \leq \|\partial^n_h( h(\id-h\a A)^{-1} )\|_{\kY^s_{n-1} \to \kY^s}
    \|  \P_m\|_{\kY^s_\ell \to \kY^s_{n-1}}\nonumber \\
    & \leq \frac{\Lambda_{n}}{\|\a\|} m^{n-1-\ell}.\label{e.hOpEst}
  \end{align}
Using  \eqref{e.hOpEst} (with $\ell$ replaced by $\ell-j$ and $n$ by
$k-j$) and 
\eqref{eqn:proj_psi_m_indep},
we can estimate the $j$-th term in the
sum of \eqref{e.dkhPsi_m} for $0\leq j\leq \ell \leq k$  as follows:
\begin{align}
\| &\partial^{k-j}_h ( h(\id-h\a A)^{-1}) \partial_h^j \P_m
  B(W_m(U,h))\|_{\kY^s}\notag \\
& \leq
\|\partial^{k-j}_h ( h(\id-h\a A)^{-1})\P_m \|_{ \kY^s_{\ell-j}\to\kY^s} \| \partial_h^j 
  B(W_m(U,h))\|_{\kY_{\ell-j}^s} \notag\\
& \leq O(m^{ k-\ell}).
\label{e.jLessEll}
\end{align}

 To obtain the first estimate of
\eqref{e.deriv-wm-psim-bound} assume that there is $b_j>0$ such that
  \begin{equation}\label{e.derivBmbound}
    \|\partial_h^j \P_m B(W_m(U,h))\|_\kY \leq b_j m^{j-\ell}
  \end{equation}
  for all $h\in [0,h_*]$, $U \in\kB^r_\ell$ and $k\geq j
  \geq \ell$.
 This will be proved below. Then, using \eqref{e.hOpEst} and \eqref{e.derivBmbound}
 we can estimate the $j$-th term in the
sum of \eqref{e.dkhPsi_m} for $j\geq \ell$ as follows:
\begin{align}
 \|\partial^{k-j}_h  & ( h(\id-h\a A)^{-1}) \partial_h^j \P_m
  B(W_m(U,h))\|_{\kY^s}\notag \\
& \leq
\|\partial^{k-j}_h ( h(\id-h\a A)^{-1})\P_m \|_{ \kY^s\to \kY^s} \| \partial_h^j 
  B(W_m(U,h))\|_{\kY^s}\notag\\
& \leq  \frac{\Lambda_{k-j}}{\|\a\|} m^{k-j} b_j m^{j-\ell}
= O(m^{k-\ell}).
\label{e.j>Ell}
\end{align}
 These estimates, with
\eqref{e.dkhPsi_m} and \eqref{e.Sm-est},  then prove
the first estimate of \eqref{e.deriv-wm-psim-bound}.

 To prove \eqref{e.derivBmbound} and the second estimate of \eqref{e.deriv-wm-psim-bound}, 
 differentiate \eqref{e.Pi} $k$ times in
  $h$: 
  \begin{equation}\label{e.derivwmInduction}
    \partial_h^k W_m = \partial_h^{k}(\id- h\a A)^{-1}\1\P_m U
    + 
    \sum_{j=0}^k{k \choose j}\partial_h^{k-j}(h\a(\id- h\a A)^{-1}\P_m) \partial_h^j B(W_m).
  \end{equation}
 By \eqref{e.Lambdak} and \eqref{eqn:proj_est}, 
  for $k\geq \ell$,
  \begin{equation}\label{e.firstterm}
  \sup_{h\in [0,h_*]}
  \| \partial_h^k(\id- h\a A)^{-1}\1\P_m \|_{\kY_\ell^s\to \kY^s} \leq
   \Lambda_k m^{k-\ell}.
  \end{equation}
  Now we show inductively  the second estimate of \eqref{e.deriv-wm-psim-bound} and
 estimate \eqref{e.derivBmbound} for  $k = \lceil \ell \rceil,\ldots,N$.  If $\ell \in \N_0$ then the start
  of the induction is $k = \ell$, and the required estimates are given
  by Theorem~\ref{l.reg-proj-num-method}. If $\ell \notin \N_0$, then the
  start of the induction is $k = \lceil \ell \rceil > \ell$.  
  If $k=  \lceil \ell \rceil$ then, due to \eqref{e.firstterm}, 
 the first term in  \eqref{e.derivwmInduction}  is of order
  $\kO(m^{k-\ell})$, and all other terms in the sum  of \eqref{e.derivwmInduction} are bounded due to
 \eqref{e.Lambdak} and \eqref{eqn:proj_psi_m_indep}   
  except when $j=k$ in the sum. Hence,  using \eqref{e.Lambda},
  \begin{equation}\label{e.derivwmest}
    \sup_{\substack{h\in [0,h_*]\\ U \in \kB^r_\ell  }}
    \| \partial_h^k W_m(U,h)\|_{\kY^s} \leq \kO(m^{k-\ell}) + 
    \Lambda \|\a\| h_* \sup_{\substack{h\in [0,h_*]\\ U \in  \kB^r_\ell }}\| \partial_h^k B(W_m(U,h))\|_{\kY^s}.
  \end{equation}

  Now we use the Fa\`a di Bruno formula \eqref{e.FaaDiBruno} again:
  \begin{equation}\label{e.Bwm}
    \partial_h^k B(W_m(U,h)) = 
    \sum_{1\leq \beta \leq k}\frac{k! \D^\beta_wB_m(W_m(U,h))}{j_1!\cdots j_k!}
      \prod_{\alpha=1}^k\left(\frac{\partial_h^\alpha W_m(U,h)}{\alpha!}\right)^{j_\alpha}
  \end{equation}
  where  $\beta=j_1+\cdots+j_k$ and the  sum is over all $j_\alpha \in N_0$, $\alpha=1,\ldots, k$ with
  $j_1+2j_2+\cdots kj_k=k$. We see that all terms on the right hand
  side of \eqref{e.Bwm} contain $h$-derivatives of order at most $k-1$
  and are therefore bounded and in particular $\kO(m^{k-\ell})$,
  except when $\beta=j_k=1$ and $j_\alpha=0$ for $\alpha\neq k$.  So we obtain
  \begin{align}
    \sup_{\substack{h\in [0,h_*]\\ U \in \kB^r_\ell   }}
    \| \partial_h^k B(W_m(U,h))\|_{\kY^s} & \leq \kO(m^{k-\ell}) +
    \sup_{\substack{h\in [0,h_*]\\ U \in \kB^r_\ell }}
    \|\D B(W_m(U,h)) \partial_h^k W_m(U,h)\|_{\kY^s} \notag \\
    & \leq \kO(m^{k-\ell}) + M'_0[R] \sup_{\substack{h\in [0,h_*]\\ U \in
        \kB^r_\ell  }} \| \partial_h^k W_m(U,h)\|_{\kY^s}.
    \label{e.Bwmderivest}
  \end{align}
  Substituting this into \eqref{e.derivwmest} gives 
the second estimate of 
  \eqref{e.deriv-wm-psim-bound} for $k = \lceil \ell \rceil$ and $h_*$ small
  enough. Resubstituting this estimate into \eqref{e.Bwmderivest}
  also shows \eqref{e.derivBmbound} for $k=\lceil \ell \rceil$.

  Now assume these estimates hold true for all $\hat k \in N_0$ with $\ell \leq \hat{k}\leq
  k-1$ and let $k\leq N$. 
  Then, using the induction hypothesis and the above estimates, in particular
\eqref{e.jLessEll}, \eqref{e.derivBmbound}, \eqref{e.j>Ell} and \eqref{e.firstterm}, all
  terms in   \eqref{e.derivwmInduction} are
  $\kO(m^{k-\ell})$ except when $j = k$ in the sum.
We deduce that
  \eqref{e.derivwmest} remains valid under the induction
  hypothesis. Moreover, by the induction hypothesis,
 each term in the sum of  the Fa\`a di Bruno formula \eqref{e.Bwm} with 
$j_k=0$ is of order
$\kO(m^n)$ with $n=0$ if $j_{\lceil \ell \rceil}+ \ldots + j_{k-1}=0$ and
\[
n = (\lceil \ell \rceil -\ell)j_{\lceil \ell \rceil}+\cdots+
(k-1-\ell)j_{k-1}\leq k-\ell
\]
if $j_{\lceil \ell \rceil}+ \ldots + j_{k-1}>0.$
 Hence \eqref{e.Bwmderivest}  remains valid, and we deduce
 \eqref{e.derivBmbound} and  the second estimate  of 
\eqref{e.deriv-wm-psim-bound}  as before.
\end{proof}


\subsection{Trajectory error for nonsmooth data }\label{ss.err_nonsmooth}

Now we are ready to prove our main result:
 
\begin{theorem}[Trajectory error for nonsmooth data]
 \label{thm:ns_convergence}
  Assume that the semilinear
  evolution equation \eqref{eqn:see} satisfies
  (A) and (B) and apply a Runge-Kutta method \eqref{eqn:rk_both} subject to 
  (RK1) and (RK2). Let
 $\ell \in I^-$, $0< \ell \leq
  p+1 $, and fix $T>0$ and $R>0$. Then there exist constants
  $h_*>0$, $c_1>0$, $c_2>0$  such that for every   $U^0 $
 with
  \begin{equation}\label{e.condPhiGlobal}
   \|\Phi^t(U^0)\|_{\kY_\ell} \leq R, \quad\mbox{for}~~t\in[0,T]
  \end{equation}
  and for all
  $h\in[0,h_*]$   we have
  \begin{equation}
  \| \Phi^{nh}(U^0) -( \Psi^h)^n(U^0)\|_\kY \leq c_1e^{c_2nh}h^{p\ell/(p+1)},
 \label{e.globTrajEst}
 \end{equation}
 provided that $nh\leq T$.
 The constants $h_*$,
  $c_1$ and $c_2$ depend only on  $R$, $T$,  \eqref{eqn:semigroup_bound},  \eqref{e.lemShA}, $\a$,
$\b$ from the numerical method
  and the bounds afforded by (B).
\end{theorem}


\noindent
{\em Proof of Theorem \ref{thm:ns_convergence}.} 
The proof consists of several steps, as outlined in the diagram below: 

\vspace*{0.3cm}
\begin{figure}[h]
\small
\begin{tabular}{ccc}
 Solution of the PDE &
Error to be estimated &
RK solution of   PDE\\[1ex]
Projection error $\downarrow$  & &  Projection error $\uparrow$ \\[1ex]
 Solution of   projected PDE &
$\longrightarrow$  &
RK solution of   projected PDE\\
 & Numerical scheme error & 
\end{tabular}
\normalsize
\end{figure}
\vspace*{0.1cm}

\noindent
We want to estimate the error of the Runge Kutta time discretization of the evolution equation (first line of the diagram). To do this, in a first step,
 we discretize in space by a Galerkin truncation. We estimate the projection error and prove regularity of the solution  $u_m(t)$ of the projected system (first column in the diagram). 
In the second step of the proof we investigate the error of the time discretization
of the space-discretized system (third row in the diagram) and couple the spatial discretization parameter $m$ with the time step size $h$ in suitable way. In the third step of the proof (third column of the diagram) we  prove regularity of the space-time discretization and estimate the  projection error of the Runge Kutta time discretization. This concludes the proof. \\

\noindent
{\em Step 1 (Regularity of solution of the projected system)}
In a first step we  aim to prove regularity  of the continuous solution of the projected system
$u_m(t) = \phi_m^t(\P_m U^0) =\Phi_m^t(  U^0)$ which will be needed later.
For the proof we denote $R$ from \eqref{e.condPhiGlobal} as $R_\Phi$ to indicate that it is a bound on $\Phi^t(U^0)$.
We will prove that there is some $r_\phi >0$
such that
\begin{equation}\label{e.hatC_ell}
\|\phi_m^t(\P_m U^0) \|_{\kY_\ell} \leq  r_\phi
\end{equation} 
uniformly in $U^0$ satisfying
\eqref{e.condPhiGlobal} and $m\geq m_*$, $t\in [0,T]$, where $m_*\geq 0$ is sufficiently
large.
Fix $\delta>0$. Then we have 
\begin{align}
\|\Phi_m^t(U^0)) \|_{\kY_\ell}  & \leq \|\P_m \Phi^t(U^0) -\Phi_m^t(U^0) \|_{\kY_\ell} + \|\P_m \Phi^t(U^0) \|_{\kY_\ell} \notag \\
& \leq m^\ell \|\P_m \Phi^t(U^0) -\Phi_m^t(U^0) \|_{\kY} 
+  \| \Phi^t(U^0) \|_{\kY_\ell} \notag \\
& \leq R_\Phi  \e^{(\omega + M')t}  + R_\Phi  =  r_\phi 
\label{e.umEll}
\end{align}
for $U^0$ satisfying \eqref{e.condPhiGlobal} and $m\geq m_*$. Here $M' = M'_0[R_\Phi+ \delta]$  and we used 
\eqref{eqn:proj_est} in the second estimate and  Lemma \ref{l.see_proj_err} and  \eqref{e.condPhiGlobal} 
in the final estimate. This proves \eqref{e.hatC_ell}.

\smallskip
\noindent
{\em Step 2 (Trajectory error of the time discretized projected system)}
Next we aim to estimate the trajectory error of the time discretization of the projected system.
 First note  that by Theorem~\ref{l.reg-proj-num-method}  (with $r$ replaced by $2r_\phi$ and
consequently $R$ by $4 r_\phi \Lambda$)
there is $h_*>0$ such that for $m \geq 0$,  $h \in [0,h_*]$
 we have $
W^i_h, \Psi^h_m \in \kCb^1(\kB^{2r_\phi}_0;\kY)
$, $i=1,\ldots, s$,
with uniform bounds in $m \geq 0$,  $h \in [0,h_*]$.
Moreover, using \eqref{e.psi}, \eqref{e.ShA} and \eqref{e.Lambda}  
  we obtain the following bound for $h \in [0,h_*]$
 to be used later:
\begin{align}
\sup_{ U \in \kB^{2r_\phi}_0}
 \| \D\Psi_m^h(U)\|_{ \kY\to\kY}  & \leq   \|\sS(hA)\|_{ \kY\to \kY} 
+ h\Lambda \|\b\| M' \|  W_m'(U)\|_{ \kY\to\kY^s}
\notag \\
& \leq   1 + \sigma h + h \Lambda  \|\b\|  M' \|  W_m'(U)\|_{\kY\to\kY^s}=: 1 + \sigma_\Psi h
\label{e.DPsimEst}.
\end{align}
where $M' = M'_0[4 r_\phi \Lambda]$.

 Now we define the global error of the projected system, for $jh\leq T$,
\begin{equation}
  E_m^j(U^0,h)=\norm{\Phi^{jh}_m(U^0)-(\Psi^h_m)^j(U^0)}{\kY}.
  \label{eqn:ge_proj}
\end{equation}
  We estimate for any $U^0$ satisfying \eqref{e.condPhiGlobal} and for all $(n+1)h\leq T$, $h\in [0,h_*]$, $m\geq m_*$,
  \begin{subequations}
    \label{e.mainComp}
    \begin{align}
      E_m^{n+1}(&  U^0, h) = \norm{\Phi^{(n+1)h}_m( U^0) -
        (\Psi^h_m)^{n+1}(U^0) }{\kY}
      \notag \\
      &\leq \norm{ \Phi^h_m(\Phi_m^{nh}(U^0)) - \Psi^h_m(\Phi_m^{nh}(U^0))
      }{\kY}
    + \norm{ \Psi^h_m(\Phi_m^{nh}(U^0)) -
        \Psi^h_m((\Psi^h_m)^{n}(U^0)) }{\kY}
      \notag \\
      &\leq \frac{h^{p+1}}{(p+1)!}\sup_{\tau\in[0,h]} \left(
        \norm{\partial_\tau^{p+1}\Phi^\tau_m(\Phi_m^{nh}(U^0))}{\kY}
        +\norm{\partial_\tau^{p+1}\Psi^\tau_m(\Phi_m^{nh}(U^0)))}{\kY} \right) \notag \\
      &  + \sup_{\theta\in[0,1]}
      \norm{\D\Psi^h_m(\Phi_m^{nh}(U^0)+\theta((\Psi^h_m)^{n}(U^0)-\Phi_m^{nh}(U^0)))}{\kY\to \kY}
      \cdot E_m^n(U^0,h) \label{e.bothSups} \\
      &\leq \frac{h^{p+1}}{(p+1)!}\left(\sup_{t\in[0,T]} 
        \norm{\partial_t^{p+1}\Phi_m^t(U^0)}{\kY}+\sup_{t\in[0,T]}\sup_{h\in[0,h_*]}
        \norm{\partial_h^{p+1}\Psi^h_m(\Phi_m^t(U^0))}{\kY}\right) \notag \\
      &\quad +
      \sup_{U\in \kB^{ 2r_\phi}_0}\norm{\D\Psi^h_m(U)}{ \kY\to \kY}\cdot
      E_m^n(U^0,h)
      \label{e.dpsiEst}\\
      &\leq \rho h^{p+1} m^{p+1-\ell} + (1+\sigma_\Psi h)E_m^n(U^0,h),
     \notag
    \end{align}
  \end{subequations}
  for some $\rho > 0$. 
 Due to \eqref{e.DPsimEst},  the second lines of \eqref{e.bothSups} and 
\eqref{e.dpsiEst} are valid as long as 
\begin{equation}\label{e.CondmeanvalueThm}
\Phi_m^{nh}(U^0)+\theta((\Psi^h_m)^{n}(U^0)-\Phi_m^{nh}(U^0))\in \kB_0^{2r_\phi}, \quad \theta \in [0,1], nh \leq T, h\in [0,h_*].
\end{equation}
Moreover the first supremum of \eqref{e.dpsiEst} is $O(m^{p+1-\ell})$  
by  Lemma \ref{l.phimDerivs}, with $R$ replaced by $r_\phi$. The second supremum  of \eqref{e.dpsiEst} is $O(m^{p+1-\ell})$    by   Lemma~\ref{lem:psim-m-dep-deriv}, with
$\kB^r_\ell$ replaced by $\kB_\ell^{r_\phi}$ (and $R$ replaced by $2r_\phi\Lambda$).

Clearly $E_m^0(U,h)=0$, so
  \begin{align*}
    E_m^n(U,h) &\leq \rho h^{p+1} m^{p+1-\ell}\frac{(1+\sigma_\Psi  h)^n-1}{\sigma_\Psi  h} \\
    &\leq \frac{\rho}{\sigma_\Psi } h^pm^{p+1-\ell}\left(1+\frac{n\sigma_\Psi  h}{n}\right)^n 
\leq \frac{\rho}{\sigma_\Psi }  h^pm^{p+1-\ell}e^{n\sigma_\Psi h}.
  \end{align*}
Choosing $m(h) = h^{-p/(p+1)}$ we see that  for $nh\leq T$, $h \in [0,h_*]$,
\begin{equation}\label{e.glob-ge_m}
  \|(\Psi^h_m)^n(U^0) - \Phi^{nh}_m(U^0) \|_{\kY} \leq
  \frac{\rho}{\sigma_\Psi }\e^{\sigma_\Psi T} h^{p}m ^{p+1-\ell  }
  = C\e^{\sigma_\Psi  T}h^{ \ell p/(p+1) }.
\end{equation}
Using  \eqref{e.glob-ge_m} we can
  ensure that   for $nh\leq T$, $h \in [0,h_*]$
  \begin{equation}
\label{e.enInY}
\|(\Psi^h_m)^n(U^0) - \Phi^{nh}_m(U^0) \|_{\kY} \leq r_\phi
\end{equation}
  by possibly reducing $h_*>0$,
and hence that \eqref{e.CondmeanvalueThm} holds.


 \smallskip
\noindent
{\em Step 3 (Projection error of numerical trajectory)}
We now  estimate  the global
projection error of the numerical method. We
will prove that for $m(h) = h^{-p/(p+1)}$, $nh\leq T$, $h \in [0,h_*]$,
\begin{equation}\label{e.glob-PsiPsihm}
  \| (\Psi^h)^n(U^0)-(\Psi_{m(h)}^h)^n(U^0)\|_{\kY} = 
\kO( m^{-\ell } )
\end{equation}
uniformly for initial data $U^0$ satisfying \eqref{e.condPhiGlobal}.

 We first establish the required regularity of the numerical trajectory of the projected system:
To bound the $\kY_{\ell}$-norm of the Galerkin truncated numerical
trajectory $(\Psi^h_{m(h)})^n(U^0)$ note that for $m=m(h) = h^{-p/(p+1)}$, 
$nh\leq T$, $h \in [0,h_*]$, with $h_*$   small enough such that $m(h_*) \geq m_*$, we have
\begin{align}
 \|(\Psi_{m(h)}^h)^n(U^0)\|_{\kY_\ell}
&\leq  \|(\Psi^h_m)^n(U^0) - \Phi^{nh}_m(U^0) \|_{\kY_\ell }+ \| \Phi^{nh}_m(U^0) \|_{\kY_\ell } \notag\\
&\leq  m^\ell \|(\Psi^h_m)^n(U^0) - \Phi^{nh}_m(U^0) \|_{\kY }+r_\phi \notag\\
&\leq m^\ell (C\e^{\sigma_\Psi T }m^{p+1-\ell} h^{p}) +r_\phi 
\leq C\e^{\sigma_\Psi T }+r_\phi
\leq r_\psi
\label{e.psim-Ell}
\end{align}
for some $r_\psi>0$. Here  $r_\phi$ is as in \eqref{e.hatC_ell} and 
we used \eqref{eqn:proj_est}  in the second line and  
\eqref{e.glob-ge_m} in the third line.

To prove \eqref{e.glob-PsiPsihm} let
\[
e^j(U^0) = (\Psi^h)^j(U^0) -(\Psi^h_m)^j( U^0)
\]
be the truncation error at time $jh \leq T$. Then for $(n+1)h\leq T$,
\begin{align}
e^{n+1}(U^0) &= (\Psi^h \circ (\Psi^h)^n)(U^0) -(\Psi^h\circ
(\Psi^h_m)^n)( U^0) \notag \\
& + (\Psi^h \circ (\Psi^h_m)^n)(U^0) -(\Psi^h_m\circ
(\Psi^h_m)^n)( U^0).
\label{e.en}
\end{align}
By Theorem \ref{l.reg-proj-num-method},  with  $r$  replaced by
$2r_\psi$  (and consequently $R$ by $4r_\psi\Lambda$, see \eqref{e.R*})  we have 
\begin{equation}
\label{e.dPsimhatDEll}
\Psi_m \in \kCb^1(\kB_{0}^{2r_\psi}; \kY).
\end{equation}
 By \eqref{e.DPsimEst}, with $\Psi_m$ replaced by $\Psi$ and the supremum taken over $\kB_{0}^{2r_\psi}$, using \eqref{e.dPsimhatDEll}
we get from \eqref{e.en} for $n\geq 1$,  $h\in [0,h_*]$ and  $(n+1)h\leq T$ that 
\begin{align}
  \|e^{n+1}(U^0) \|_{\kY}
&\leq   \sup_{\theta \in [0,1]  } 
 \| \D\Psi^h( (\Psi^h_m)^n   + \theta ((\Psi^h)^n  -(\Psi^h_m)^n )(U^0) )\|_{ \kY\to \kY}\|  e^n(U^0)\|_{\kY} \notag \\
&
+ \| e^1((\Psi^h_m)^n(U^0))\|_{\kY}  \notag \\
 &\leq    \sup_{\|U\|_{\kY_\ell} \leq  2r_\psi  } 
 \| \D\Psi^h(U)\|_{ \kY\to \kY}\|  e^n(U^0)\|_{\kY} 
+ \| e^1((\Psi^h_m)^n(U^0))\|_{\kY}   \notag\\
  & \leq   (1+\sigma_\Psi h)\|  e^n(U^0)\|_{\kY} + h \kO(m^{-\ell}),
\label{e.enRecursion}
\end{align}
where $m=m(h)$, 
with order constant uniformly in all $U^0$ satisfying \eqref{e.condPhiGlobal},  as long as 
\begin{equation}
\label{e.psimPsiTheta}
(\Psi^h_m)^n(U^0)  + \theta ((\Psi^h)^n(U^0) -(\Psi^h_m)^n(U^0) ) \in
\kB_0^{ 2r_\psi}, \quad \theta \in [0,1].
\end{equation}
Here we used that for $U \in \P_m\kY$,
\[
e^1(U) = h\b^T(\id-h\a A)^{-1}\left( \left( \P_m( B(W(U,h))-B(W_m(U,h)) \right) + \Q_m B(W(U,h)) \right),
\]
so that for $U \in \kB_{\ell}^{r_\psi} \cap \P_m\kY$, $h\in [0,h_*]$, by
\eqref{e.wwm} (with $r$ replaced by   $r_\psi$ and $R$ by $2r_\psi\Lambda$)
\begin{align} \| e^1(U)\|_{\kY }
  & \leq  h \|\b\| \Lambda ( M'\|W(U,h)-W_m(U,h)\|_{\kY^s}  +\|\Q_m B(W(U,h))\|_{\kY^s})  \notag \\
  & \leq  h \|\b\|( \Lambda M'\kO(m^{-\ell}) + \kO(m^{-\ell}) ) =  h
  \kO(m^{-\ell}),
\label{e.e1Est}
\end{align}
where $m=m(h)$ and  $M'=M_0'[2r_\psi\Lambda]$. In the last inequality of \eqref{e.e1Est} we used that 
\begin{equation}
\label{e.BmW}
\|\Q_m B(W(U,h))\|_{\kY^s}  \leq m^{-\ell} M = \kO( m^{-\ell}),
\end{equation}
where  $M=M_0[2r_\psi\Lambda]$.

From \eqref{e.enRecursion} we deduce for $nh\leq T$, $h\in [0, h_*]$ and all $U^0$ satisfying \eqref{e.condPhiGlobal} that
\begin{align}
  \|   e^n(U^0)\|_{\kY} &\leq (1+\sigma_\Psi h)^{n-1}\|  e^1(U^0)\|_{\kY} +  
  \frac{1}{\sigma_\Psi h}\left((1+\sigma_\Psi h)^{n-1}-1\right)h\kO(m^{-\ell})
  \notag\\
  &\leq    \exp(\sigma_\Psi T)(\|  e^1(U^0)\|_{\kY} + \kO(m^{-\ell })) 
  = \kO(m^{-\ell}),\label{e.enEst}
\end{align}
with $m=m(h)$.
Here \eqref{e.e1Est} does not apply to $\| e^1(U^0)\|_\kY $ because in general $U^0 \notin \P_m \kY$.
But from  \eqref{e.psipsim}  we see 
 that $\| e^1(U^0)\|_\kY = O(m^{-\ell})$.
By choosing a possibly bigger $m_*$ (and, by virtue of $m=h^{-p/(p+1)}$, a smaller $h_*$)
  we can achieve that
$\|e^n(U^0)\|_{\kY} \leq r_\psi$ so  that the
required condition \eqref{e.psimPsiTheta}  is
satisfied. This proves \eqref{e.glob-PsiPsihm}.

Hence, \eqref{e.globPhiGalError}, \eqref{e.glob-ge_m} and \eqref{e.glob-PsiPsihm}
prove that
\begin{equation}\label{e.EnEst}
E^n(U^0,h) = E_m^n(U^0,h) + \kO(m^{-\ell}) = \kO( h^{ p\ell/( p+1)  } )
\end{equation}
for $nh\leq T$, $h\in [0, h_*]$ and $U^0$ satisfying \eqref{e.condPhiGlobal}.
\qed


\begin{example}\label{ex:cubicNSE} (Cubic nonlinear Schr\"odinger equation in $\R^3$)
We now consider a cubic nonlinear Schr\"odinger equation in $\R^3$
\begin{equation}\label{e.NSER3}
\i u_t = \Delta u + |u|^2 u
\end{equation}
as in \cite{LubichNSE}. We rewrite it in the form \eqref{eqn:see} with $U = (u_1,u_2)$
where $u = u_1+\i u_2$ with  
\[
A= \left(\begin{array}{cc}
0  &\Delta\\
-\Delta & 0\\
\end{array}
\right),\quad\mbox{and}\quad
 B(U ) = (u_1^2 + u_2 ^2){u_2 \choose -u_1},
\]
cf.~also Example \ref{ex:NSE}, 
and consider it on   $\kY=\kH_{2}(\R^3;\R^2)$.  
By  Lemma \ref{l.nonlinearityH^l} a) the nonlinearity  $B(U)$ is analytic on $\kY$ 
and the same  holds true on 
$\kY_\ell= D(A^\ell) =\kH_{2(\ell+1)}(\R^3,\R^2) $
where $\ell\geq 0$. In this case assumption (B) holds
for $I=[0, L]$ and any $L>0$. If  \eqref{e.NSER3} is discretized by the implicit mid point rule  and $U^0 \in \kY_1=\kH_4$, 
 then from 
Theorem \ref{thm:ns_convergence}  we obtain 
an order of convergence $\kO(h^{2/3})$ in the
$\kH_2$-norm.  In \cite{LubichNSE} a second order Strang type time discretization is used 
to discretize  \eqref{e.NSER3} and a better rate of convergence is observed, namely
an order of convergence $\kO(h)$ in the
$\kH_2$-norm for $U^0 \in \kH_4$.  This is due to the fact that the linear part of the 
evolution equation  \eqref{eqn:see}, i.e., 
$\dot U = AU$, is integrated exactly by this method. We plan to extend the methods of this paper
to splitting and exponential integrators in future work.
\end{example}

\section{Appendix: Trajectory error on  general domains}
\label{s.appendix}
In this appendix we show how to extend the results of this paper to more general domains.
We make the following assumption for the nonlinearity $B(U)$ of the semilinear evolution equation
\eqref{eqn:see}:
 
\begin{enumerate}[(B1)]
\item There exists  $L\geq 0$, $I \subseteq [0,L]$,
$0,L \in I$,
$N\in \N$, $N >\lceil L\rceil  $ and a  nested collection of open,
  $\kY_\ell$-bounded sets $\kD_\ell \subset \kY_\ell$, $\ell \in I$, such that $B\in
  \kCb^{N- \lceil \ell \rceil }(\kD_\ell;\kY_\ell)$ for
 $\ell\in I$.
  \label{enum:B1}
\end{enumerate}
   Similarly as before we denote the supremum  of
$B:\kD_\ell\to\kY_\ell$ as $M_\ell$ and the supremum of its   derivative as
 $M'_\ell$,  and set $M =M_0 $, $M' =M'_0$ and $\kD=\kD_0$.  
  
 The right hand side of the evolution equation  \eqref{eqn:see} is bounded in the $\kY$ norm for 
$U \in  \kD \cap D(A) $ and it is well-known that there exists a  differentiable solution $\Phi^t(U) \in \kY$ in this case, see \cite{P83} and Theorem \ref{t.semiflow-gen} below.  Extending this setting  we will in this section consider initial data $U^0 \in \kY_\ell$ with $\ell \in J^-$ defined as follows: 
\begin{equation}\label{e.I-Gen}
J^-:= \{ \ell \in J: \ell -k\in I, k =1,\ldots, \lfloor\ell \rfloor\}, \quad\mbox{where}\quad   J = I \cup [L, L+1],
\end{equation}
 similarly as in  \eqref{e.I}.
%
For our main result, Theorem \ref{thm:ns_convergence-gen} below, we  need an additional condition on the nonlinearity $B$ of  \eqref{eqn:see}.
\begin{enumerate}[(B2)]
\item 
$B: \kD_{[\ell-1]^+}  \cap \kB_\ell^{ R}\to \kY_\ell$
is bounded for any $\ell>0$   with $\ell\in J^-$ and any $R>0$.
\label{enum:B'}
\end{enumerate}

Here we define $[x]^+ = \max(x,0)$ for $x\in \R$.
Assumption (B2) is often satisfied for superposition operators, see 
Lemma  \ref{l.nonlinearityH^l} b), c) and  in particular Example  \ref{ex:SWEnonAnalytic} where
the potential $V$ of the semilinear wave equation is only defined on an open subset $D$ of $\R$. 

For a subset $\kU$ of some Hilbert space  $\kY$ and $\delta>0$ we denote  by
\[
\kU^\delta = \bigcup_{u \in \kU} \kB_\kY^\delta(u)
\]
a $\delta$-neighbourhood of $\kU$.
Moreover  for any  subset
 $\kU_\ell$ of $\kY_\ell$, $\ell \in I$, we define $(\kU_\ell)^\delta_\ell = \kU_\ell^\delta$
 as a $\delta$-neighbourhood of $\kU_\ell$ in $\kY_\ell$.
In the following let $ \kU_\ell \subseteq \kD_\ell$, $\ell \in I$, be a nested
collection of open sets and $\delta>0$ be such that
\begin{equation}\label{e.U_ell}
\kU_\ell^{\delta}\subseteq \kD_\ell, \quad \ell \in I.
\end{equation}
We will also frequently use the abbreviation
\begin{equation}\label{e.hatU}
   \widehat\kU_{\ell}  :=\kU_{[\ell-1]^+}  \cap \kB_\ell^{R}
\end{equation}
for $\ell \in J$.
  
To extend Theorem \ref{t.semiflow} (and also Theorem \ref{thm:num_method_regularity}, see below) to general domains we cover the domain
$\kU_\ell$ with open balls of radius $\delta$ and apply the corresponding theorems on  each ball. To ensure uniformity 
of the maximal
time interval of existence $T_*$ we consider initial data  in $(\widehat\kU_\ell)^{\delta/2}_0$.

\begin{theorem}[Regularity of the semiflow on general domains]\label{t.semiflow-gen} 
Assume   (A) and (B1) and  choose 
 $\ell >0$. Then
  there is $T_*>0$ such that 
\begin{subequations}
  \begin{equation}
    \Phi^t\in \kCb^{N}((\widehat\kU_\ell)^{\delta/2}_0;\kD) 
   \label{e.PhiUDeriv-gen}
  \end{equation}
with uniform bounds in  $t\in [0,T_*]$. Moreover if $\ell \in J^-$ and 
 $k\in \N_0$ satisfies
$k \leq \ell$,  then
\begin{equation}
 \Phi(U) \in \kCb^k([0,T_*];\kD) 
 \label{eqn:semiflow_regt-gen}
\end{equation}
with uniform bounds in $U \in \widehat\kU_\ell$.
\label{e.semiflow_reg-gen}
\end{subequations}
  The bounds on $T_*$ and  $\Phi$,   depend only on $\delta$ from \eqref{e.U_ell}, $R$ from 
  \eqref{e.hatU},  $\omega$ from \eqref{eqn:semigroup_bound}, and those afforded by
  assumption (B1).
  \end{theorem}
  \begin{proof}
  The proof is a  modification of the proof of Theorem \ref{t.semiflow}.
   Here we  let $U^0 \in \widehat\kU_\ell$ and take $R=\delta$. As before we  compute $W$ as
  fixed point of the map $\Pi$ from \eqref{eqn:see_mild},  but this time we consider $\Pi$ as map from $\kB^{\delta}_{\kZ}(U^0) \times \kB^{\delta/2}_\kY(U^0)\times [0,T_*]$ to $
  \kZ$ noting that by \eqref{e.U_ell} we have  $\kB^{\delta}_{\kY}(U^0)\subseteq \kD$.
   Then  \eqref{e.ProblemTerm} becomes
   \begin{equation}\label{e.ProblemTermNew}
     \|\Pi(W,U,T) - U^0 \|_{\kZ} \leq
\max_{\tau \in [0,1]} \|\e^{\tau TA} (U-U^0)\|_{\kY}  +\max_{\tau \in [0,1]}\|(\e^{\tau TA} -1)U^0\|_{\kY}
 + T \e^{\omega T} M
   \end{equation}
and, for $0<\epsilon\leq \min(1,\ell)$,  we estimate the additional term as follows
\[
 \max_{\tau \in [0,1]}   \|(\e^{\tau T A}-\id)   U^0\|_{\kY}
 \leq 
  \max_{\tau \in [0,1]}  \|(\e^{\tau T A}-\id) \|_{\kY_\epsilon\to \kY}
    \|U^0\|_{\kY_\epsilon} 
    \leq c T^\epsilon \|U^0\|_{\kY_\epsilon} 
 \]
   uniform in 
$U^0 \in \widehat\kU_{\ell}$. Here we have used Lemma
  \ref{l.etA-epsilon} below and that $\|U^0\|_{\kY_\epsilon}\leq \|U^0\|_{\kY_\ell}\leq R$  for  $U^0 \in \widehat\kU_{\ell}$.  Hence for $T_*>0$ sufficiently,  $\Pi(\cdot, U,T)$ maps $\kB^{\delta}_{\kZ}(U^0)$  into itself and, similarly as before, $W\in
  \kCb(\kB^{\delta/2}_{\kY}( U^0) \times [0,T_*];\kB^{\delta}_{\kZ}( U^0))$
with $N$ derivatives in the first component  and uniform bounds in 
$U^0 \in \widehat\kU_{\ell}$.  This proves \eqref{e.PhiUDeriv-gen}.

 Note that the term $ \|(\e^{\tau T A}-\id)  U^0\|_{\kY}$ in
  \eqref{e.ProblemTerm} can not be made small uniformly in $U \in \kU_0$
  since the operator $\e^{tA}$ is not uniformly continuous
  in $t$. But we can make that term order $\kO(T^\epsilon)$ uniformly
  in $U^0 \in \widehat\kU_\ell$ due to Lemma \ref{l.etA-epsilon} below.
 
 The proof of \eqref{eqn:semiflow_regt-gen} is similar to the analogous result \eqref{e.semiflow_reg} on balls, with obvious modifications. 
  \end{proof}

 The following lemma was needed in the proof:
\begin{lemma}
  \label{l.etA-epsilon}
  Assume (A). Then for every $T_*>0$ there is
  some $c>0$ such that for $\epsilon \in [0,1]$, $T\in [0,T_*]$,
 \begin{equation}\label{e.etA-epsilon}
    \| \e^{TA}-\id \|_{\kY_\epsilon\to \kY} \leq c T^\epsilon.
  \end{equation}
\end{lemma}
\begin{proof}
We have with $m(T) =  1/T$
\begin{align*}
 \| \e^{TA}-\id \|_{\kY_\epsilon\to \kY} &  \leq 
 \| \P_m (\e^{TA}-\id) \|_{\kY_\epsilon\to \kY}  +  \| \Q_m (\e^{TA}-\id) \|_{\kY_\epsilon\to \kY}\\
& \leq  \| \P_m \int_0^T A \e^{tA} dt  \|_{\kY_\epsilon\to \kY} +  \| \Q_m (\e^{TA}-\id) \|_{\kY_\epsilon\to \kY}\\
&  \leq  \e^{T\omega} \| \P_m T A  \|_{\kY_\epsilon\to \kY}  
+(1+ \e^{T\omega})  \| \Q_m  \|_{\kY_\epsilon\to \kY}\\
&  \leq  \e^{T\omega} T m^{1-\epsilon}
+(1+ \e^{T\omega}) m^{-\epsilon} =  (1+2 \e^{T\omega}) T^{\epsilon}.
\end{align*}
Here we used \eqref{eqn:semigroup_bound} in the third line,  \eqref{eqn:proj_est}  in the last line
and we estimated, using \eqref{e.A_ell}, that
\[
 \| \P_m  A  \|_{\kY_\epsilon\to \kY}  \leq \| \P_m  A^{1-\epsilon} \|_{\kY\to\kY }  \|A^\epsilon\|_{\kY_\epsilon\to\kY } \leq m^{1-\epsilon}.
\]
\end{proof}

  
  \begin{theorem}[Regularity of numerical method on general domains]
  Assume   (A),  (B1),  (RK1) and (RK2) and let
 $\ell >0$.  
 Then there is $h_*>0$  such that
  \begin{subequations}
\begin{equation}\label{eqn:glob-U-num_method_regularity-gen}
    W^i(\cdot,h), \Psi(\cdot,h) \in \kCb^{N} ((\widehat\kU_{\ell})^r_0;\kD),
  \end{equation}
with uniform bounds in  $h \in [0,h_*]$.
Here  $r=r(\delta)$ is as in \eqref{e.R*}.
 Furthermore, for
  $\ell \in J^-$, $k\in \N_0 $, $k \leq \ell$,
we have for $i=1,\ldots, s$,
\begin{equation}\label{eqn:glob-num_method_regularity-gen}
    W^i(U,\cdot), \Psi(U,\cdot) \in \kCb^k( [0,h_*];\kD)
  \end{equation}
\end{subequations}
  with uniform bounds in $U \in\widehat\kU_\ell$. 
  The bounds on   $h_*$, $\Psi$ and  $W$   depend only on $\delta$ from \eqref{e.U_ell}, 
  $R$ from   \eqref{e.hatU}, \eqref{e.lemShA}, those afforded by assumption (B1) and
 on $\a$, $\b$  as specified by the
  numerical method.
  \label{thm:num_method_regularity-gen}
\end{theorem}

\begin{proof}
To prove  
 \eqref{eqn:glob-U-num_method_regularity}   let  $U^0 \in \widehat\kU_\ell$. 
As in  the proof of Theorem \ref{thm:num_method_regularity} we   compute $W$ as
  fixed point of the map $\Pi$ from \eqref{e.Pi-W}, but this time we consider $\Pi$ as a map from  $ \kB^{r}_{\kY^s}(\1 U^0)\times \kB^\delta_{\kY}(U^0)\times [0,h_*]$ to
   $\kY^s$ where $r=r(\delta)$ is as in \eqref{e.R*}.
   To  check that $\Pi(W,U,h) \in  \kB^{\delta}_{\kY^s}(\1 U^0)$ for sufficiently small $h_*>0$ let $0<\epsilon\leq  \min(1,\ell)$. 
   Then 
\begin{equation} 
 \|\Pi(W,U,h) - \1 U^0 \|_{\kY^s} \leq
\Lambda \| U-U^0\|_{\kY^s} +
    \|((\id-h \a A)^{-1}-\id) \1 U^0\|_{\kY^s} + h  \norm{\a}{}\Lambda M, \label{e.Pi-Fix-Gen}
 \end{equation}
 and we estimate 
 \begin{align}\label{e.ProblemTerm2}
 \|((\id-h \a A)^{-1}-\id) \1 U^0\|_{\kY^s}   &\leq \|((\id-h \a A)^{-1}-\id) \|_{\kY_\epsilon^s\to \kY^s}
    \|U^0\|_{\kY_\epsilon}   \leq c h^\epsilon \|U^0\|_{\kY_\epsilon} 
  \end{align}
  for $h\in [0,h_*]$ and $h_*$ small enough and  independent of 
$U^0 \in \widehat\kU_{\ell}$. Here we have used Lemma
  \ref{l.Lambda-epsilon} below and that $\|U^0\|_{\kY_\epsilon} 
\leq \|U^0\|_{\kY_\ell} \leq R$ for $U^0 \in \widehat\kU_{\ell}$.
The other terms of \eqref{e.Pi-Fix-Gen} are estimated as in \eqref{e.Pi-Fix}
with $R$ replaced by $\delta$.
So $\Pi$ maps $\kB^{\delta}_{\kY^s}(\1 U^0)$ to itself and is a
  contraction for $h_*$ small enough.  This proves statements
\eqref{eqn:glob-U-num_method_regularity-gen} and also
\eqref{eqn:glob-num_method_regularity-gen}  in the case $k=0$. 

  Note that the term  in
  \eqref{e.ProblemTerm2} can not made small independent of $U \in \kU_0$
  since the operator $(\id-h \a A)^{-1}$ is not uniformly continuous
  in $h$. But we can make that term order $\kO(h^\epsilon)$ uniformly
  in $U^0 \in \widehat\kU_\ell$ due to Lemma \ref{l.Lambda-epsilon} below.
 
  The rest of the proof is similar to the proof of Theorem \ref{thm:num_method_regularity}.
\end{proof}

In the proof we needed the following lemma:

\begin{lemma}
  \label{l.Lambda-epsilon}
  Assume (A), (RK1) and (RK2). Then there are
  $h_*>0$, $c>0$ such that for $\epsilon \in [0,1]$, $h\in [0,h_*]$,
 \begin{equation}\label{e.Lambda_h_epsilon}
    \|(\id-h \a A)^{-1}-\id \|_{\kY_\epsilon^s\to \kY^s} \leq c h^\epsilon
  \end{equation}
\end{lemma}
\begin{proof}
  By Lemma \ref{l.ShA} there is $h_*>0$ such that $(\id-h\a A)^{-1}$
  is bounded as map from $\kY^s$ to itself, uniformly in $h\in
  [0,h_*]$.
  Note that
  \[
  \|(\id-h \a \P A)^{-1}-\id \|_{\kY_\epsilon^s\to \kY^s} \leq c h 
  \]
  with $\P = \P_1$ as before. Due to the definition of the norm on $\kY_\ell$, see \eqref{eqn:inner_prod},
   it remains to prove
  that  
  \begin{equation}\label{e.QA_h_epsilon}
    \| f(L, \a,h) \|_{\kY^s\to \kY^s} \leq c h^\epsilon,
  \end{equation}
  where $L = (\id-\P) A$ and  
  \[
  f(\lambda,\mu, h) :=
  \lambda^{-\epsilon}((1-h \mu\lambda )^{-1}-1).
  \]
Because $L$ is normal \eqref{e.QA_h_epsilon} is equivalent to 
\begin{equation}\label{e.heps-spec}
 \sup_{ \lambda \in \spec(L)} \| f(\lambda, \a, h) \|_{\C^s\to \C^s} \leq c h^\epsilon.
 \end{equation}
  Let $\mu$ be an eigenvalue of $\a$.  We first show that
  \begin{equation}\label{e.hnu}
  \sup_{ \lambda \in \spec(L)}   | f(\lambda, \mu,h)  |  \leq c h^\epsilon.
  \end{equation} 
  Note that $\epsilon\leq 1$ and so
$0$ is  a removable singularity of $f(\cdot, \mu,h)$.  Furthermore the pole $\lambda_0(h) = 1/(h\mu)$ satisfies $\Re\lambda_0(h)>0$
because $\Re(\mu)>0$  for all $\mu \in \spec(\a)$ by  (RK2).  
  By \eqref{eqn:semigroup_bound} there is $\omega \in \R$ with
$
\Re \spec(L)   \leq \omega.
$
For $h_*>0$ sufficiently small  we have $\Re\lambda_0(h)>\omega$ 
for all $h \in [0,h_*]$ and so $\lambda_0(h) \notin \spec L$. Moreover a straightforward
computation shows that there is $\delta>0$ such that for sufficiently small $h_*>0$
\[
|1-h \mu\lambda|\geq \delta \quad\mbox{for all}\quad \lambda \in \spec(L), \mu \in \spec(\a), h \in [0,h_*].
\]
For example it is sufficient to choose 
\[
\delta < \Re \mu/|\mu|-h_*\max(\omega,0) |\mu| \quad\mbox{for all}\quad \mu \in \spec(\a).
\]
 Then 
$f(\cdot,\mu, h)$ is continuous
on $D = \{ \lambda \in \C,~|1-h\lambda\mu |\geq \delta\}$.
Now let $z= h\mu \lambda$ and define 
\[
g:\Omega\to \C, \quad\mbox{where}\quad g(z) := z^{-\epsilon}((1-z)^{-1}-1)\quad\mbox{and}\quad 
 \Omega =  \{ z \in \C, |z-1|\geq \delta\}.
 \]
Then $g:\Omega\to \C$ is continuous and   $f(\lambda, \mu, h) =   (h\mu)^{\epsilon}g(h\mu\lambda) $.
Since  $\lim_{z
    \to \infty} g(z) =0$ and a continuous function is bounded on a compact set,
$g$ is bounded on $\Omega$, uniformly in $\epsilon \in [0,1]$.  
  That proves \eqref{e.hnu}.  If $\a$ is diagonalizable then
  \eqref{e.hnu} implies \eqref{e.heps-spec} and \eqref{e.QA_h_epsilon}. 
  
  Now consider the case
  where $\a$ has Jordan blocks and $\mu$ is an eigenvalue of $\a$ with
  differing algebraic  and geometric multiplicity. 
  Let $n$  be its algebraic  multiplicity. Let $ E_\mu$ be the
generalized eigenspace of  $\a$ to the
  eigenvalue $\mu$.
  Then we can find coordinates on $ E_\mu$ such that
  \[
\a|_{E_\mu} = \a_\mu =   \mu {\bf 1} + N,
  \]
where  ${\bf 1}$ is the identity on $\C^n$ and $N$ is a nil-potent  $(n,n)$-matrix, i.e., $N^n={\bf 0}$
(the $(n,n)$ null-matrix).
   Then it is sufficient to prove
\eqref{e.heps-spec} with $\a$ replaced by $\a_\mu$ for all $\mu \in \spec(\a)$.
We have
\begin{align}
f(\lambda,\a_\mu,  h)   = \lambda^{-\epsilon} ( ( {\bf 1} - h \a_\mu \lambda)^{-1}- {\bf 1})  = 
(h\mu)^\epsilon G(z,N) 
\label{e.heps-jordan}
\end{align}
where 
\begin{align*}
G(z,N ) & =  z^{-\epsilon}\left( \left( {\bf 1} -  z \left({\bf 1} + \frac{N}{ \mu} \right) \right)^{-1} -{\bf 1} \right)  
 = \frac{1}{z^{\epsilon}
 (1-z)} \left( {\bf 1}  -   \frac{z}{ \mu(1-z)}  N  \right)^{-1} - \frac{{\bf 1} }{z^{\epsilon}}  \\
& = \sum_{j=0}^{n-1} \frac{   z^{j-\epsilon}}{\mu^j(1-z)^{j+1} }N ^j - z^{-\epsilon}{\bf 1}  
 =  \sum_{j=1}^{n-1} \frac{   g_j(z)}{\mu^j  }N ^j  + g(z) {\bf 1}
\end{align*}
and we  
set  $z=h\mu \lambda$,
 $g_j(z) := z^{j-\epsilon}/(1-z)^{j+1}$. Here we used the geometric series and  the fact that $N^n={\bf 0}$.
The functions $g_j(z):\Omega \to \C$ are
continuous and
$\lim_{z\to \infty} g_j(z) = 0$ for all $j=1,\ldots, n-1$, and the same is true for $g(z)$.  Therefore, as before
$g,g_j:\Omega \to \C$ are bounded uniformly in $\epsilon \in [0,1]$. With \eqref{e.heps-jordan} this shows \eqref{e.heps-spec} and hence   \eqref{e.Lambda_h_epsilon}.
\end{proof}

The following lemma is an adaptation of Lemma \ref{l.proj-semiflow}
 and Lemma   \ref{l.see_proj_err} to the setting considered in this
section:

\begin{lemma}[Regularity of projected semiflow and projection error on general domains]
Assume (\ref{enum:A}) and (B1), let $\delta>0$ be as in  \eqref{e.U_ell} and let  $\ell >0$. 
 Then  there is  
  $m_*\geq 0$ such that for $m\geq m_*$ there
  exists a projected semiflow $\Phi_m$  with the properties specified in Theorem \ref{t.semiflow-gen}, 
with uniform bounds in
$m\geq m_*$.
%
Moreover  choose $T>0$. Then for sufficiently large $m_*\geq 0$ the following holds:
  for all $U^0$ with
 \begin{equation}\label{e.PhiGalErrorCond-gen}
    \Phi^t(U^0) \in \widehat\kU_\ell, ~~t \in [0,T]
  \end{equation}
 and  for all $m\geq m_*$ we have
  $\Phi_m^t(U^0) \in \kD$ for $t\in [0,T]$, and
 \eqref{e.globPhiGalError} is true
  with an order constant that depends only on $\delta$, $R$ from \eqref{e.hatU}, $T$,
  \eqref{eqn:semigroup_bound} and the bounds afforded by
 (B1).
 \label{l.see_proj_err-gen}
\end{lemma}
\begin{proof} The only modification required to apply Theorem \ref{t.semiflow-gen}
 is that we need to choose $m_*(\delta)\geq 0$ large enough
to be able to apply the contraction mapping theorem on $\P_m \Pi(W, \P_m U, h)$, with $\Pi$ as in \eqref{eqn:see_mild}, see \cite{OW10B}. The proof of \eqref{e.globPhiGalError}  is similar to the proof of Lemma \ref{l.see_proj_err}, with obvious modifications.
\end{proof}
 
 
\begin{lemma}[Regularity of projected numerical method
and projection error on general domains]
\label{l.reg-proj-num-method-gen}
 Assume   (A), (B1),  (RK1) and (RK2),  let $\delta>0$ be as in  \eqref{e.U_ell}  and 
 let $\ell >0$.  Then there is $m_*\geq 0$ such that 
  $W^i_m$, $i=1,\ldots, s$ and $\Psi_m$ satisfy \eqref{eqn:glob-U-num_method_regularity-gen} and, if $\ell \in J^-$, also   \eqref{eqn:glob-num_method_regularity-gen}
with uniform bounds in  
$m\geq m_*$.
%
%
Moreover,  if  $\ell \in J^-$,  then  \eqref{e.wwm} and
\eqref{e.psipsim} hold true for
$m\geq m_*$,  with $\kB_\ell^r$ replaced by $\widehat\kU_\ell$. 
 The bounds on   $h_*$, $m_*$,  $\Psi_m$ and $W_m$   and the order constants 
depend only on $\delta$, $R$ from \eqref{e.hatU},
   \eqref{e.lemShA}, the bounds afforded by assumption (B1)  and
 on $\a$, $\b$  as specified by the
  numerical method.
\end{lemma}
\begin{proof}
The proof is a modification of the proof of Lemma \ref{l.reg-proj-num-method}.
To prove \eqref{eqn:glob-U-num_method_regularity-gen} and   \eqref{eqn:glob-num_method_regularity-gen} 
for the projected numerical method  we need to choose $m_*\geq 0$ large enough
to be able to apply the contraction mapping theorem on $\P_m \Pi(W, \P_m U, h)$, with $\Pi$ as in \eqref{e.Pi-W},
 see \cite{OW10B}. 

To prove \eqref{e.wwm} in this setting, we need
estimate the term in the second  line of \eqref{e.WWm-est} differently than in \eqref{e.WWm-est} because from (B1) 
we can not guarantee
that $W^i(U,h) \in \kD_\ell$, $i=1,\ldots, s$; in particular this is wrong if $\ell>L$. Therefore  we cannot estimate $B(W)$ in the $\kY_\ell^s$ norm. We proceed 
 as follows: note that, since $\ell \in J^-$ there is      $\epsilon \in (0,1]$   such that $\ell-\epsilon \in I$. Then by \eqref{eqn:glob-U-num_method_regularity-gen}, with $\kD$ replaced by $\kD_{\ell-\epsilon}$, 
there is $h_*>0$ such that for $h\in [0,h_*]$,
 $W^i(\cdot, h) \in \kCb(\widehat\kU_{\ell}, \kD_{\ell-\epsilon})$, $i=1,\ldots, s$.
  Hence
  \begin{align*}
  \norm{h \a (\id-h\a A)^{-1} \Q_mB(W)}{\kY^s}  & \leq 
   \|\Q_m\|_{  \kY^s_{\ell}\to \kY^s}
\| h \a (\id - h \a A)^{-1} \|_{    \kY^s_{\ell-\epsilon}\to\kY^s_{\ell}    }  
 \| B(W)\|_{\kY^s_{\ell-\epsilon}}\\
 & \leq  \Gamma M_{\ell-\epsilon} m^{-\ell}
  \end{align*}
   with an order constant uniform in $U\in \widehat\kU_{\ell}$.
 Here we used \eqref{e.Gamma}  which will be proved in Lemma \ref{l.Gamma}
 below. Then solving \eqref{e.WWm-est}  for $\|W(U,h)-W_m(U,h)\|_{\kY^s}$ gives   \eqref{e.wwm}. 

To prove \eqref{e.psipsim} in this setting we estimate the term  $\| \Q_m \b h(\id - h \a A)^{-1} B(W)\|_{\kY^s}$ in the first line of \eqref{e.psipsimHelp} as follows:
\begin{align*}
 \| \Q_m \b h(\id - h \a A)^{-1} B(W)\|_{\kY^s} &
\leq s\|\b\| \|\Q_m  \a^{-1} h\a (\id - h \a A)^{-1} \|_{\kE(\kY_{\ell-\epsilon}^s,\kY^s)} \| B(W) \|_{\kY_{\ell-\epsilon}^s}
\\
& \leq s\|\b\| \|\Q_m\|_{\kY_{\ell}\to \kY} 
\| \a^{-1} h\a (\id - h \a A)^{-1} \|_{\kE(\kY_{\ell-\epsilon}^s,\kY^s_{\ell})} M_{\ell-\epsilon}\\
& \leq  s\|\b\| \| \a^{-1} \| \Gamma  M_{\ell-\epsilon} m^{-\ell}. 
 \end{align*}
 Inserting this into \eqref{e.psipsimHelp} proves \eqref{e.psipsim}, with  $\kB_\ell^r$ replaced by $\widehat\kU_\ell$.
\end{proof}
  
 The following lemma was needed in the proof:
  \begin{lemma}\label{l.Gamma}
 Under assumptions (A), (RK1) and (RK2) let $h_*>0$,
 and  $\Lambda>0$   be as in Lemma \ref{l.ShA}.  Then for $h\in
  [0,h_*]$
\begin{equation}
\| h \a (\id - h \a A)^{-1} \|_{    \kY^s\to \kY^s_{1}  }
\leq \Gamma := h_*\|\a\| \Lambda + (\Lambda +1) .
\label{e.Gamma}
\end{equation}
  \end{lemma}
  \begin{proof}
 Using Lemma \ref{l.ShA} we estimate
\begin{align*}
\| h \a (\id - h \a A)^{-1} \|_{  \kE( \kY^s  , \kY^s_{ 1}) }  
&\leq \| h \a (\id - h \a A)^{-1} \|_{  \kE( \kY^s ) }
+ \| h \a A (\id - h \a A)^{-1} \|_{  \kE( \kY^s) } \notag\\
&\leq h\|\a\| \Lambda + (\Lambda +1) . 
\end{align*}
  \end{proof}
  
  \begin{lemma}[$m$-dependent bounds for derivatives of $\Phi_m$ on general domains]\label{l.phimDerivs-gen}
Assume (A) and (B1) and 
choose $\ell>0$ with 
$\ell\in J^-$, $T> 0$, $R>0$ and $m_*\geq 0$.
 Then \eqref{eqn:globproj_flow_reg} holds true  for all $k\in\N_0$ with $k\leq \ell$
  and  all $U^0$ with 
\begin{equation}\label{e.glob-cond-proj_flow_reg-gen}
  \|\Phi^t_m(U^0)\|_{\kY_\ell}\leq R,\quad
  \Phi^t_m(U^0) \in \kD_{[\ell-1]^+}\quad\mbox{for}\quad t\in [0,T],~~m\geq m_*.
  \end{equation}
Further,  for $k\in \N_0$ with
    $\ell\leq k \leq N$ and for all $U^0$ satisfying
  \eqref{e.glob-cond-proj_flow_reg-gen} the estimate
\eqref{eqn:proj_flow_m_dep}
is still true
  with bounds uniform in $U^0$.   The bounds and order
constants only depend on $T$,  \eqref{eqn:semigroup_bound},   $R$ from \eqref{e.glob-cond-proj_flow_reg-gen}
and the bounds from assumption (B1).
\end{lemma}
\begin{proof}
The proof is similar to the  proof of Lemma \ref{l.phimDerivs}, but with 
$\kB^R_{\ell-k}$ replaced
 $\kD_{\ell-k}$ in \eqref{eqn:proj_flow_k_m_indep}.
 \end{proof}

  \begin{lemma}[$m$-dependent bounds for derivatives of $\Psi_m$ and $W_m$ on general domains]
  Assume   (\ref{enum:A}), (B1),   (RK1) and (RK2). Choose 
 $\ell\in J^-$, $\ell>0$  and
$k\in \N_0$,
  $\ell <k\leq N$.
Then there are
  $h_*>0$ and $m_*\in\N$ such that  for $m\geq m_*$ 
and    $i=1,\ldots, s$, 
  \eqref{e.wmpsim} holds, with $\kB^R_0$ replaced by $\kD$, 
  with $m$-dependent bounds which are uniform in 
$U \in \widehat\kU_\ell  $.
Moreover  \eqref{e.deriv-wm-psim-bound}
holds true with $\kB_\ell^r$ replaced by  $\widehat\kU_\ell $.
The  order constants in  \eqref{e.deriv-wm-psim-bound} depend only on $\delta$ from \eqref{e.U_ell}, $R$ from \eqref{e.hatU}, \eqref{e.lemShA}, $\a$ and $\b$ from the numerical method 
and the bounds afforded by (B1). 
 \label{lem:psim-m-dep-deriv-gen}
\end{lemma}
\begin{proof}
The following modifications have to be made to the proof of Lemma \ref{lem:psim-m-dep-deriv}:
replace $\kB^R_{\ell-j}$  by $\kD_{\ell-j}$ and $\kB_\ell^r$  by $\widehat\kU_\ell $.
Furthermore, when $j=0$, $k>\ell>0$ in \eqref{e.jLessEll} then 
we do not know if $W_m^i \in \kD_\ell$,  in particular this is wrong if $\ell>L$. Therefore
we cannot use (B1) to bound $\|B(W_m(U,h))\|_{\kY^s_\ell}$. So  we proceed by
 Lemma \ref{l.reg-proj-num-method-gen}, with $\kY$ replaced by $\kY_{[\ell-1]^+}$, to obtain
  that $W_m(U,h) \in \kD_{[\ell-1]^+}$
for $U \in \widehat\kU_\ell  $, $h\in [0,h_*]$ which with  \eqref{e.hOpEst}  implies 
\begin{align*}
 \|  \partial^{k }_h  & ( h(\id-h\a A)^{-1})  \P_m B(W_m(U,h))\|_{\kY^s} \leq  \\
& \quad\quad\quad \|  \partial^{k }_h ( h(\id-h\a A)^{-1}) \P_m \|_{ \kY^s_{[\ell-1]^+}\to \kY^s } 
\| B(W_m(U,h))\|_{\kY^s_{ [\ell-1]^+}} \\
& \quad\quad\quad
= O(m^{k-1 - [\ell-1]^+   }  )  
 \leq O(m^{k-\ell}).
\end{align*}

\end{proof}
 
 \begin{theorem}[Trajectory error for nonsmooth data on general domains]
 \label{thm:ns_convergence-gen}
  Assume (A), (B1), (B2), 
  (RK1) and (RK2),  and let   
 $\ell \in J^-$ with $0< \ell \leq
  p+1 $.  Fix $T>0$. Then there exist constants 
  $h_*>0$, $c_1>0$, $c_2>0$ such that for every   $U^0 $
 with
  \begin{equation}\label{e.condPhiGlobal-gen}
    \{\Phi^t(U^0):t\in[0,T]\}\subset\widehat\kU_{\ell}
  \end{equation}
  and for all
  $h\in[0,h_*]$  
 estimate  \eqref{e.globTrajEst} holds true for $nh\leq T$.
 The constants $h_*$,
  $c_1$ and $c_2$ depend only on   $\delta$ from \eqref{e.U_ell}, $R$ from \eqref{e.hatU},  
  $T$,  \eqref{eqn:semigroup_bound},  \eqref{e.lemShA}, $\a$,
$\b$ from the numerical method
  and the bounds afforded by (B1) and (B2).
\end{theorem}
\begin{proof}
The main difference to the
 proof of Theorem \ref{thm:ns_convergence} is that we have to ensure that the Galerkin truncation and
 time discretization of the solution $\Phi^t(U^0) $ stay in the domain $\kD_{[\ell-1]^+}$.

In the first step, where we  prove regularity of $\Phi^t_m(U^0) = u_m(t)$, we make the following changes:
 we first apply Lemma \ref{l.see_proj_err-gen}
with $\kY$ replaced
by $ \kY_{[\ell-1]^+}$, with $\ell$ replaced by $\epsilon:=\ell - [\ell-1]^+\in (0,1]$, with $\kD$ replaced by
$\kU_{[\ell-1]^+}^{\delta/4}$  and with $\delta$ from \eqref{e.U_ell} replaced by $\delta/4$. 
This, together with  \eqref{e.hatC_ell} shows that, for sufficiently large $m_*$, we have  for all $U^0$ satisfying
\eqref{e.condPhiGlobal-gen}  
\begin{equation}\label{e.phimHatDEll}
\Phi_m^t(U^0) \in  \kU_{[\ell-1]^+}^{\delta/4} \cap \kB_\ell^{r_\phi}, ~~t \in [0,T], m\geq m_*.
\end{equation}

 In the second step of the proof we make the following changes:
  in this case,  due to Lemma \ref{l.reg-proj-num-method-gen}, we have $
\Psi^h_m \in \kCb^1((\widehat\kU_\ell)^r_0;\kD)
$
for $r=r(\delta)$ as in \eqref{e.R*}.
 Using Lemma \ref{l.see_proj_err-gen} we get  for $m\geq m_*$ with  $m_*$ sufficiently large 
 that 
\begin{equation}\label{e.PhimDEllr/2}
\Phi^t_m(U^0) \in (\widehat\kU_{\ell})^{r/2}_0  ~~\mbox{for}~~ t\in [0,T].
\end{equation}
This ensures that $\Psi_m^h$ is well defined on the trajectory $u_m(t) = \Phi^t_m(U^0) $ of the
Galerkin truncated system. Moreover $\Psi_m^h$ is well-defined on the numerical trajectory
$(\Psi^h_m)^n(U^0)$, $nh\leq T$, as long as 
\begin{equation}
\label{e.enInY-gen}
\|(\Psi^h_m)^n(U^0) - \Phi^{nh}_m(U^0) \|_{\kY} \leq r/2,
\end{equation}
which will be proved later.
Furthermore estimate \eqref{e.DPsimEst} on $\D \Psi^h_m$  holds with $\kB^{2r_\phi}_0$ replaced by $(\widehat\kU_\ell)^r_0$. Also, \eqref{e.mainComp} holds with the same replacement in \eqref{e.dpsiEst} and in \eqref{e.CondmeanvalueThm}. 
In this case  the first term of \eqref{e.bothSups} is $O(m^{p+1-\ell})$  
by  Lemma \ref{l.phimDerivs-gen} (with $R$ replaced by $r_\phi$) which applies due to \eqref{e.phimHatDEll}. The second term  of \eqref{e.bothSups} is $O(m^{p+1-\ell})$    by   Lemma~\ref{lem:psim-m-dep-deriv-gen}, with
$\widehat\kU_\ell$ replaced by $ \kU_{[\ell-1]^+}^{\delta/4}\cap \kB_\ell^{r_\phi} $   
(and consequently $\delta$  in   \eqref{e.U_ell} replaced by $3\delta/4$), see \eqref{e.hatU}  
and \eqref{e.phimHatDEll}.
Using  \eqref{e.glob-ge_m} we can achieve \eqref{e.enInY-gen}  for $h_*>0$ small enough.
This then ensures that conditon 
\eqref{e.CondmeanvalueThm}, with $\kB^{2r_\phi}_0$ replaced by $(\widehat\kU_\ell)^r_0$, is satisfied.

 In the third step of the proof we make the following changes:
 We first establish the required regularity of the numerical trajectory of the projected system: we prove that
\begin{equation}
\label{e.psihmDEll-1}
(\Psi^h_m)^n(U^0) \in \kU_{ [\ell-1]^+}^{\delta/2}
~~\mbox{for}~~ m=m(h), nh \leq T, h \in [0,h_*]
\end{equation}
and all $U^0$ satisfying
\eqref{e.condPhiGlobal-gen}. 
From  \eqref{e.glob-ge_m}  and \eqref{eqn:proj_est} we see that, with 
$\epsilon = \ell - [\ell-1]^+$, for $h \in [0,h_*]$, $nh\leq T$,
\begin{align}
\| (\Phi^{nh}_m(U^0)-(\Psi_{m(h)}^h)^n(U^0)\|_{\kY_{ \ell-\epsilon }}
& \leq  \|\Phi^{nh}_m(U^0) -(\Psi^h_m)^n(U^0)  \|_{\kY} m^{\ell-\epsilon}\notag \\
& \leq  C\e^{\sigma_\Psi T }h^{ \ell p/(p+1) - (\ell-\epsilon)p/(p+1) }
 =
C\e^{\sigma_\Psi T}h^{ \epsilon p/(p+1)}.\notag
\end{align}
Using that $\Phi^t_m(U^0) \in\kU_{[\ell-1]^+}^{\delta/4}$
for $t\in [0,T]$, see \eqref{e.phimHatDEll}, we can achieve \eqref{e.psihmDEll-1}, 
 by possibly decreasing $h_*>0$.
 The estimate \eqref{e.psim-Ell}, together with  \eqref{e.psihmDEll-1}, shows that
\begin{equation}
\label{e.pshmHatDell}
(\Psi^h_{m(h)})^n(U) \in \widehat\kU_{\ell}^{\delta/2}:=\kU_{[\ell-1]^+}^{\delta/2}\cap \kB_\ell^{r_\psi},\quad 0 \leq nh \leq T, h \in [0,h_*].
\end{equation}
%
%
By Lemma \ref{l.reg-proj-num-method-gen}, with $\widehat \kU_\ell$ replaced by $ \widehat\kU_{\ell}^{\delta/2}$, and consequently $\delta$ replaced by $\delta/2$ and $r=r(\delta)$  replaced by
$r(\delta)/2$ (see \eqref{e.R*})  we have 
\begin{equation}
\label{e.dPsimhatDEll-gen}
\Psi_m \in \kCb^1((\widehat\kU_{\ell}^{\delta/2})^{r/2}_0; \kD).
\end{equation}
Then \eqref{e.enRecursion} is still valid, with  the condition that $\|U\|_{\kY_\ell} \leq  2r_\psi$ replaced by
the condition $U \in (\widehat\kU_{\ell}^{\delta/2})^{r/2}_0$ and with an order constant which is uniform in 
all $U^0$ satisfying \eqref{e.condPhiGlobal-gen} provided that \eqref{e.psimPsiTheta} holds with 
$\kB_0^{ 2r_\psi}$ replaced by
$(\widehat\kU_{\ell}^{\delta/2})^{r/2}_0$.

Moreover \eqref{e.e1Est} holds for 
$U \in \widehat\kU_{\ell}^{\delta/2} \cap \P_m\kY$
instead of  $U \in \kB_{\ell}^{2r_\psi} \cap \P_m\kY$,
 by Lemma \ref{l.reg-proj-num-method-gen},   with $\widehat\kU_{\ell}$ replaced by $\widehat\kU_{\ell}^{\delta/2}$.
To prove \eqref{e.BmW} in this setting note  that by \eqref{eqn:glob-U-num_method_regularity-gen}, 
with
$\kD$ replaced by $\kD_{[\ell-1]^+ }$, 
with $\widehat\kU_{\ell}$ replaced by $\widehat\kU_{\ell}^{\delta/2}$ and with  $\ell$ by
$ \epsilon := \ell - [\ell-1]^+$,
we have 
\begin{equation}
W^i(U,h) \in \kD_{ [\ell-1]^+ }, ~~\mbox{for}~~ U \in \widehat\kU_{\ell}^{\delta/2},
h \in [0,h_*],  i=1,\ldots, s.
\label{e.WD-1}
\end{equation}
Moreover, by \eqref{e.Pi}, \eqref{e.WD-1}, Lemma \ref{l.ShA} and \eqref{e.Gamma} we have for $i=1,\ldots, s$, $h\in [0,h_*]$ and  $ U \in \widehat\kU_{\ell}^{\delta/2}$  
that
\[
\|W^i(U,h)\|_{\kY_\ell} \leq R_W := \Lambda r_\psi  + \Gamma M_{[\ell-1]^+}.
\]
This shows that
$W^i(U,h) \in \kD_{[\ell-1]^+} \cap \kB^{R_W}_\ell$,  $i=1,\ldots, s$.
Hence   by (B2),  we have 
\[
\sup_{\substack{h \in [0,h_*]\\ U \in \widehat\kU_{\ell}^{\delta/2}}} \|  B(W(U,h))\|_{\kY_\ell^s}  <\infty
\]
and that, with \eqref{eqn:proj_est}, proves \eqref{e.BmW}.

The local projection error $e^1(U)$ along the numerical trajectory $(\Psi^h_m)^n(U^0)$ is estimated
as in the proof of Theorem \ref{thm:ns_convergence}.
 By possibly choosing a bigger $m_*$   we can achieve that
$\|e^n(U^0)\|_{\kY} \leq r/2$, $nh\leq T$, $h\in [0,h_*]$.
Then, due to \eqref{e.pshmHatDell},  the
required condition \eqref{e.psimPsiTheta}, with $\kB_0^{ 2r_\psi}$ replaced by
$(\widehat\kU_{\ell}^{\delta/2})^{r/2}_0$,  is
satisfied. This proves \eqref{e.glob-PsiPsihm} and concludes the proof.
\end{proof}



\section*{Acknowledgments}
The authors want to thank the Nuffield foundation for the Summer
Bursary Scheme under which this project was started in 2009.


\end{document}